\documentclass[11pt, letterpaper]{article}
\usepackage[utf8]{inputenc}

\title{Lower Bounds for 
Parallel and Randomized Convex Optimization}
\author{Jelena Diakonikolas\thanks{Partially supported by the NSF grant \#CCF-1740855. Part of this work was done while the author was a Microsoft Research Fellow at the Simons Institute for the Theory of Computing, for the program on Foundations of Data Science, and while she was a postdoctoral researcher at Boston University.}\\[5pt]
Department of Statistics\\[8pt]
UC Berkeley\\[8pt]
\texttt{jelena.d@berkeley.edu}
\and 
Crist\'{o}bal Guzm\'{a}n\thanks{Partially supported by the FONDECYT project 11160939, and the Millennium
Science Initiative of the Ministry
of Economy, Development, and
Tourism, grant “Millennium Nucleus Center for the
Discovery of Structures in Complex Data.”}\\[5pt]
Pontificia Universidad Cat\'{o}lica de Chile\\
Millennium Nucleus Center for the\\
Discovery of Structures
in Complex Data\\
\texttt{crguzmanp@mat.uc.cl}} 

\date{}
\usepackage{times}
\usepackage{cite}
\usepackage{amsfonts}
\usepackage{amsmath}
\usepackage{amsthm}
\usepackage{amssymb}
\usepackage{comment}
\usepackage{dsfont}
\usepackage{color}
\usepackage{bm}
\usepackage{url}
\usepackage[colorinlistoftodos,prependcaption]{todonotes}
\usepackage{thmtools}
\usepackage{thm-restate}
\usepackage{xfrac}
\usepackage{soul}
\usepackage[colorlinks,urlcolor=blue,citecolor=blue,linkcolor=blue]{hyperref}

\usepackage[margin=1in]{geometry}

\usepackage{graphicx}

\usepackage{array}
\usepackage{multirow}
\newcolumntype{M}[1]{>{\centering\arraybackslash}m{#1}}
\newcolumntype{N}{@{}m{0pt}@{}}

\newtheorem{theorem}{Theorem}[section]
\newtheorem{proposition}[theorem]{Proposition}
\newtheorem{lemma}[theorem]{Lemma}
\newtheorem*{mainthm}{Main Theorem}

\theoremstyle{definition}
\newtheorem{definition}[theorem]{Definition}

\newtheorem{claim}[theorem]{Claim}

\theoremstyle{remark}
\newtheorem{remark}[theorem]{Remark}

\numberwithin{equation}{section}

\newcommand{\vecspace}{\mathbf{E}}
\newcommand{\feasset}{{\cal X}}

\newcommand{\RR}{\mathbb{R}}
\newcommand{\prob}{\mathbb{P}}
\newcommand{\EE}{\mathbb{E}}
\newcommand{\eps}{\varepsilon}

\newcommand{\bx}{\boldsymbol{x}}
\newcommand{\bX}{\boldsymbol{X}}
\newcommand{\bh}{\boldsymbol{h}}
\newcommand{\bg}{\boldsymbol{g}}
\newcommand{\bz}{\mathbf{z}}
\newcommand{\br}{\mathbf{r}}
\newcommand{\by}{\boldsymbol{y}}
\newcommand{\bxi}{\bm{\xi}}
\newcommand{\blambda}{\boldsymbol{\lambda}}

\newcommand{\dd}{\mathrm{d}}

\newcommand{\Compl}{\mbox{Compl}}
\newcommand{\innp}[1]{\left\langle #1 \right\rangle}
\newcommand{\opnorm}[1]{\left\| #1 \right\|_{\mathrm{op}}}

\newcommand{\mI}{\mathbf{I}}

\begin{document}

\maketitle
\thispagestyle{empty}
\begin{abstract}
We study the question of whether parallelization in the exploration of the feasible set can be used to speed up convex optimization, in the local oracle model of computation. We show that the answer is negative for both deterministic and randomized algorithms applied to essentially any of the interesting geometries and nonsmooth, weakly-smooth, or smooth objective functions. In particular, we show that it is not possible to obtain a polylogarithmic (in the sequential complexity of the problem) number of parallel rounds  with a polynomial (in the dimension) number of queries per round. In the majority of these settings and when the dimension of the space is polynomial in the inverse target accuracy, our lower bounds match the oracle complexity of sequential convex optimization, up to at most a logarithmic factor in the dimension, which makes them (nearly) tight. 
Prior to our work, lower bounds for parallel convex optimization algorithms were only known in a small fraction of the settings considered in this paper, mainly applying to Euclidean ($\ell_2$) and $\ell_\infty$ spaces. 
Our work provides a more general approach for proving lower bounds in the setting of parallel convex optimization. 
\end{abstract}
\newpage
\setcounter{page}{1}
%
%
%%%%%%%%%%%%%%%%%%%%%%%%%%%% INTRODUCTION
\section{Introduction}\label{sec:intro}

Convex optimization 
has been successfully applied to obtain
faster algorithms for solving linear systems~\cite{ST04,cohen2014solving,KOSZ13}, network flow problems~\cite{KLOS2014,Sherman2013,sherman2017area}, continuous relaxations of discrete problems, such as positive linear programs~\cite{d-allen2014using,d-wang2015unified} and submodular optimization~\cite{chekuri2015multiplicative}. In machine learning, accelerated methods~\cite{Nesterov1983,axgd,beck2009fast} %
are some of the workhorses of supervised learning through empirical risk minimization. 
Given the scale of modern datasets resulting in extremely large problem instances, an attractive approach to reducing the time required for performing computational tasks is via parallelization. Indeed, many classical problems in theoretical computer science are well-known to be solvable in polylogarithmic number of rounds of parallel computation, with polynomially-bounded number of processors. For examples in submodular optimization we refer to~\cite{BalkanskiSubmod:2018}, and even continuous approaches via the multilinear relaxation may be found in~\cite{Chekuri:2018,Ene:2018}.

When it comes to convex optimization, parallelization is in general highly beneficial in computing local function information (at a single point from the feasible set), such as its gradient or Hessian, and can generally be exploited to improve the performance of optimization algorithms. However, a natural barrier for further speedups is parallelizing the exploration of the feasible set. This leads to the following question: \\[5pt]
{\centerline{\emph{Is it possible to improve the oracle complexity of convex optimization via parallelization?} 
}}\\[5pt]
Here, oracle complexity is defined as the number of adaptive rounds an algorithm needs to query an arbitrary oracle providing local information about the function, such as, e.g., its value, gradient, Hessian, or a Taylor approximation at the queried point from the feasible set, before reaching a solution with a specified accuracy. Note that most of the commonly used optimization methods, such as, e.g., gradient descent, mirror descent and its special case -- multiplicative weights updates, Newton's method, the ellipsoid method, Frank-Wolfe, and Nesterov's accelerated method, all work in this local oracle model. 

Beyond 
its potential use as a generic way to accelerate convex optimization, parallelization in the exploration of the feasible set would also impact other areas. In stochastic convex optimization, gradient descent applied to the empirical risk can be seen as an adaptive data analysis procedure. Recent developments in this area~\cite{Dwork:2015} provide sample complexity bounds for this algorithm, from an application of differential privacy. Here, a reduction in the number of adaptive rounds could lead to improved sample complexity bounds for the performance of (optimization via the) ERM, as results in this literature are only mildly affected by parallel queries, whereas there is much higher sensitivity to the number of adaptive rounds.

The study of parallel oracle complexity of convex optimization was initiated by Nemirovski in the early '90s~\cite{Nemirovski:1994}. In this work, 
it was shown that for nonsmooth Lipschitz-continuous optimization over the $\ell_\infty$ ball, it is not possible to attain  polylogarithmic parallel round complexity with polynomially many processors. However, the lower bound does not match the sequential complexity $\Theta(d\ln(1/\varepsilon))$; the author conjectured that the parallel complexity in this case should be $\Omega(d\ln(1/\eps)/\ln(K))$.
Since the work of Nemirovski~\cite{Nemirovski:1994} and until very recently, there has been no further progress on this conjecture, 
nor in obtaining lower bounds for other settings, such as, e.g., weakly/strongly smooth optimization over more general feasible sets. 
Very recently, motivated by the applications in online learning, local differential privacy, and adaptive data analysis, several lower bounds for parallel convex optimization \emph{over the Euclidean  space} have been obtained~\cite{smith2017interaction,balkanski2018parallelization,nips2018-woodworth,duchi2018minimax} (for a more detailed description of these results, see Section~\ref{sec:related-work}). 
Our main result shows that it is not possible to improve the oracle complexity of convex optimization via parallelization, for deterministic or randomized algorithms,  
different levels of smoothness, and essentially all interesting geometries -- $\ell_p$ balls for $p \in [1, \infty]$, together with their matrix spectral analogues, known as Schatten spaces, $\mbox{Sch}_p$. 
The resulting lower bounds are robust to enlargements of the feasible set, and thus apply in the unconstrained case as well. This is a much more general result than previously addressed in the literature, where similar results were obtained only for (i) \emph{constrained Euclidean} ($\ell_2$) setups~\cite{smith2017interaction,balkanski2018parallelization,nips2018-woodworth,duchi2018minimax} and (ii) the \emph{nonsmooth} $\ell_\infty$ setup that only applies to deterministic algorithms~\cite{Nemirovski:1994}. 
The results for non-Euclidean settings require novel ideas and surpassing several technical challenges. Further, these settings are of fundamental interest. For example, $\ell_1$-setups naturally appear in sparsity-oriented learning applications; $\mbox{Sch}_{1}$ (a.k.a.~nuclear norm) appears in matrix completion problems~\cite{Nemirovski:2013}; finally, smooth $\ell_{\infty}$-setups have been used in the design of fast algorithms for network flow problems~\cite{Lee:2013,KLOS2014}. %Our results imply that none of these algorithms can be further accelerated through the use of parallelization. 

\subsection{Our Results}

Our results rule out the possibility of improvements by parallelization, showing that, in high dimensions, 
 sequential 
methods are already optimal for any amount of parallelization that is polynomial in the dimension.\footnote{Ruling out parallelization via an exponential number of queries is unlikely, since such a high number of queries would, in general, allow an algorithm to construct an $\eps$-net of the feasible set and choose the best point from it.} 
Our approach is to provide a generic lower bound for parallel oracle algorithms and use reductions between different classes of optimization problems. 
Below, $\varepsilon>0$ is the target accuracy, $K$ is the number of parallel queries per round, and $d$ is the dimension. 

\begin{mainthm}
(Informal)
Unless $K$ is exponentially large in the dimension $d$, any (possibly randomized) algorithm working in the local oracle model and querying up to $K$ points per round, when applied to the following classes of convex optimization problems over $\ell_p$ balls and 
$\mbox{Sch}_p$ balls:
\begin{itemize}
    \itemsep0em 
    \item Nonsmooth (Lipschitz-continuous) minimization for $1 < p < \infty$ and $d = \Omega(\mathrm{poly}(\frac{1}{\eps^{p + p/(p-1)}}))$;
    \item Smooth (Lipschitz-continuous gradient) minimization for $2\leq p \leq \infty$ and $d = \Omega(\mathrm{poly}(\frac{1}{\eps}))$;
    \item {Weakly-smooth (H\"{o}lder-continuous gradient) minimization} for $2\leq p \leq \infty$ and $d = \Omega(\mathrm{poly}(\frac{1}{\eps}))$
\end{itemize}
takes asymptotically at least as many rounds to reach an $\eps$-approximate solution as it would take without any parallelization, up to, at most, a $\sfrac{1}{\ln(d)}$ factor.
\end{mainthm}
\noindent As mentioned before, our result easily extends to unconstrained optimization over $\ell_p$ normed spaces. The small subset of the possible cases not included in the theorem are off by small factors and are still very informative: they rule out the possibility of any significant improvement in the round complexity via parallelization (see Table~\ref{tab:LB_summary}  and the discussions in Sections~\ref{sec:techniques} and~\ref{sec:apps}). 

{To present the results in a unified manner, we use the definition of weakly-smooth functions, 
i.e., functions with $\kappa$-H\"older-continuous gradient, which 
interpolates between the classes of nonsmooth  
($\kappa = 0$) and smooth functions ($\kappa = 1$) (see Section~\ref{sec:prelims} for a precise definition).} 
These two  
special cases 
are summarized in Table~\ref{tab:LB_summary}. For the precise statements encompassing the weakly-smooth cases ($\kappa \in (0, 1)$) as well as the specific {\em high-dimensional} regime for $d$, see Section~\ref{sec:apps}. 
\begin{table}[t] 
\centering
\begin{tabular}{|M{1.8cm} | M{2cm} | M{2.2cm} | M{3.5cm} | M{3.4cm} | N}
\hline\small
Function class	 & $p=1$ & $1< p <2$  & $2\leq p<\infty$ & $p=\infty$  &\\[5pt] \hline
\small Nonsmooth ($\kappa=0$) & $\Omega\big(\frac{1}{\varepsilon^{2/3}}\big)$ 
&   $\Omega\big(\frac{1}{\varepsilon^2}\big)$     &             
$\Omega\big(\frac{1}{\varepsilon^p}\big)$             &        
$\Omega\Big(\big(\frac{\varepsilon^2 d}{\ln(d K/\gamma)}\big)^{\sfrac{1}{3}}\Big)$ $(\ast)$        %  
&\\[5pt]\hline
\small Smooth ($\kappa=1$) &  $\Omega\Big(\frac{1}{\ln(d)\varepsilon^{2/5}}\Big)$ &$\Omega\Big(\frac{1}{\ln (d) \varepsilon^{2/5}}\Big)$
& $\Omega\Big(\frac{1}{\min\{p, \ln(d)\}\varepsilon^{{p}/{(p+2)}}} \Big)$
&  $\Omega\Big(\frac{1}{\ln(d)\varepsilon}\Big)$ &\\[5pt] \hline
\end{tabular}
\caption{\small High probability lower bounds for parallel convex optimization, in the $\ell_p^d$ and $\mbox{Sch}_p^d$ setups. Here, $d$ is the dimension, $\varepsilon$ is the accuracy, $K$ is the number of parallel queries per round, and $1-\gamma$ is the confidence. Except for $(\ast)$, the high dimensional regime requires $d=\Omega(\mbox{poly}(1/\varepsilon,\ln(K/\gamma)))$. %
}\label{tab:LB_summary}%\vspace{-25pt}
\end{table}

{The largest gap obtained by our results is in the nonsmooth $\ell_1$-setup. 
Here, the $\Omega(1/\varepsilon^{2/3})$ bound comes from a reduction from the $\ell_{\infty}$ case, which explains the discontinuity in the first row of the table. We also consider a non-standard setting of $\ell_p$-Lipschitz nonsmooth optimization for $p \in [1, 2)$ over an $\ell_2$ ball inscribed in the unit $\ell_p$ ball. Even in this smaller domain the complexity is 
$\Omega(1/\varepsilon^2)$, which provides a strong evidence of   higher complexity for the $\ell_1$-setting.}
\subsection{Overview of the Techniques}\label{sec:techniques}
%\todo[inline]{This requires a full update, but will get back to it.}
Most of the lower bounds for large-scale convex optimization in the literature (e.g.,~\cite{Nemirovski:1983,Guzman:2015,nips2018-woodworth,balkanski2018parallelization}) are based on the construction of a hard problem instance defined as the maximum of affine functions. Each affine function $f_i$ is defined by its direction vector $\bz^i$ and offset $\delta_i$. Examples of the direction vectors that are typically used in these works include signed orthant vectors, uniform vectors from the unit sphere and scaled Rademacher sequences. At an intuitive level, a careful choice of these affine functions prevents any algorithm from learning more than one direction vector $\bz^i$ per adaptive round. At the same time, an appropriately chosen set of affine functions ensures that the algorithm needs to learn \emph{all} of the vectors $\bz^i$ before being able to construct an $\eps$-approximate solution. We take the same approach in this paper.

Most relevant to our work are the recent lower bounds for parallel convex optimization over Euclidean ($\ell_2$) spaces~\cite{balkanski2018parallelization,nips2018-woodworth}, which are tight in the large-scale regime. In these works, the argument about learning one vector $\bz^i$ at a time is derived by an appropriate concentration inequality, while the upper bound on the optimal objective value is obtained from a good candidate solution, built as a combination of the random vectors. In terms of learning a single vector $\bz^i$ per round, a simple application of the Hoeffding inequality suffices for the $\ell_2$ setups. %However, in the more general $\ell_p$ setups,  this in general does not provide sufficiently strong concentration. 
%This is discussed in more details below for a specific example $p=1.$
%
%Providing good candidate solutions in the $\ell_2$ or $\ell_{\infty}$ settings can be done explicitly. However, for more general geometries there is no obvious strategy in this direction. 
%
However, there is no obvious way of generalizing the lower bounds for the Euclidean setting to the more general $\ell_p$ geometries. For example, in the $\ell_p$-setup for $p>2$, these arguments only lead to a lower bound of $\Omega(1/\varepsilon^2)$, which is far from the sequential complexity $\Theta(1/\varepsilon^p)$ (further discussed below). 
On the other hand, the use of relationships between the $\ell_p$ norms leads to uninformative lower bounds. In particular, for $p \in [1, 2),$ the appropriate application of inequalities relating $\ell_p$ norms needs to be done for both feasible sets (relating $\|\cdot\|_p$ and $\|\cdot\|_2$) and the Lipschitz (or smoothness) constants (relating $\|\cdot\|_{p^*}$ and $\|\cdot\|_2,$ where $p^* = \frac{p}{p-1}$). Unless $p \approx 2,$ this approach leads to a degradation in the lower bound by a polynomial factor in $d$. For example, if $\ell_2$ case is used to infer a lower bound for the $\ell_1$ setup, the resulting lower bound would be of the order $1/({d}\eps^2)$ and $1/(d\sqrt{\varepsilon})$, for nonsmooth and smooth cases, respectively. Such quantities are are far from the sequential lower bounds $\Omega(\ln(d)/\eps^2)$ and $\Omega(1/\sqrt{\eps})$ applying to the nonsmooth and smooth settings, respectively.

%\textcolor{blue}
{Our lower bounds are based on families of 
%One of our main contributions to address non-Euclidean setups is a method to provide upper bounds on the value of random objectives, by the use of duality and anti-concentration bounds.
%We provide constructions of 
random vectors $\bz^1,\ldots,\bz^M$ which: (i) satisfy concentration along their marginals, so the ``learning one vector per round'' argument applies; and (ii)  lead to a large negative optimal value via a minimax duality argument. Each particular regime will require different constructions of the random vectors, that we describe below. However, all lower bounds will be obtained from a general result (Theorem \ref{thm:MetaThm}) that shows that (i) and (ii) suffice to get a lower bound for parallel convex optimization and completely streamlines the analysis.}
%\todo[inline]{Say more about MetaTheorem?}

\paragraph{Main Ideas in the Nonsmooth Setting.}
%\todo[inline]{We should decide whether to use a different example here now, or just summarize the general approach.}
%\textcolor{blue}
{As mentioned before, the main distinction in our constructions comes from the value of $p$.
In the case $p\leq 2$, it suffices to use dense random vectors, such as (scaled) Rademacher sequences. For this choice, certifying (i) is straightforward via Hoeffding's Inequality; however, as $p$ decreases the concentration becomes weaker, which for the limit case $p=1$ simply does not hold. On the other hand, property (ii) is established by an upper bound obtained via minimax duality. The optimality gap is determined by the infimum of the $\ell_{p^*}$ norm of convex combinations of the scaled Rademacher vectors; namely, by $\inf_{\lambda\in\Delta_M}\|\sum_{i\in[M]} \lambda_i\bz^i\|_{p^*}$. This quantity is proved to be $\Omega(\|\lambda\|_2)=\Omega(1/\sqrt{M})$ w.h.p., by the Khintchine Inequality together with a net argument. Thus, the choice  $M=O(1/\eps^2)$ leads to the desired optimality gap of $\Omega(\eps)$. 

On the other hand, the case $p>2$ has a rather different behavior. It is possible to observe that scaled Rademacher sequences (or random vectors on the unit $\ell_p$-sphere) only provide a $M=\Omega(1/\varepsilon^2)$ lower complexity bound; this happens again due to the Khintchine Inequality. To resolve this issue, we enforce disjointness of the supports of vectors $\bz^i$ (whose non-zero elements are independent scaled Rademacher random variables). This careful choice of vectors $\bz^i$ allows us to satisfy (ii) in a relatively straightforward manner. On the other hand, property (i) is guaranteed by choosing the supports of each $\bz^i$ large enough, which provides the necessary concentration of inner products.}

\paragraph{Main Ideas in the Smooth and Weakly Smooth Settings.}

To obtain lower bounds for smooth and weakly-smooth objectives in the case $p \geq 2$, we use infimal convolution smoothing from~\cite{Guzman:2015} applied to the following objective function:
\begin{equation} \label{eqn:hard_function} f(\bx)=\max\left\{\frac12\max_{i\in[M]}[\langle\bz^i,\bx\rangle-i\delta], \|\bx\|_p-R\right\}, 
\end{equation}
where $R>0$ is chosen so that within the $\|\cdot\|_p$-unit ball  the left term dominates. This added $\|\bx\|_p-R$ term in the definition of the objective functions makes our lower bounds robust to the enlargements of the feasible set, and further allows us to obtain results in the setting of $p \in [1, 2)$ discussed below.

The choice of our smoothing is not arbitrary: first, we need the smoothed function to be uniformly close to the nonsmooth one; second, we need to obtain the smallest possible smoothness constant on the objective function; and finally, we need the smoothing to respect the local behavior of $f(\cdot)$, so that the ``learning one vector per round'' argument applies. We show that this is possible by coupling our generic lower bound stated in Theorem~\ref{thm:MetaThm} with the smoothing from~\cite{Guzman:2015}. 
This approach leads to tight lower bounds for the smooth $\ell_p$-setup, for any constant $p \geq 2$, and nearly-tight lower bounds for $p$ that is larger than a constant. 

This type of smoothing is not directly applicable in the $p \in [1, 2)$ cases, as a consequence of the fact that there are no known regularizers for an infimal convolution smoothing.  This is also related to the fact that $\ell_p$ spaces for $p \in [1, 2)$ are not $2$-uniformly smooth (see, e.g.,~\cite{Ball:1994}),  which leads to a natural barrier for the approach. To overcome this barrier, we use the approach from~\cite{Guzman:2015}, which allows reducing the smooth $p \in [1, 2)$ cases from the $p=\infty$ case. While this approach leads to non-trivial lower bounds (see Table~\ref{tab:LB_summary} and Theorem~\ref{thm:1<p<2-smooth}), these lower bounds are not tight in general. For example, in the case of large-scale smooth minimization, the lower bound scales with  $\frac{1}{(\ln(1/\eps) + \ln\ln(K/\gamma))\eps^{2/5}}$. Apart from the logarithmic factor, this bound is off by a factor $1/\eps^{1/10}.$ This is a direct consequence of our lower bounds not being sufficiently tight in the low-dimensional  regime of the $\ell_\infty$ setting. We note that improving these low-dimensional $\ell_\infty$ lower bounds has been an open problem since the work of Nemirovski~\cite{Nemirovski:1994}. 

%
%
%%%%%%%%%%%%%%%%%%%%%%%%%%%% RELATED WORK
\subsection{Related Work}\label{sec:related-work}

%\todo[inline]{We should mention somewhere that the setting in Balkanski Singer is non-standard ($\ell_2$ optimization over the $\ell_\infty$ ball).}

As mentioned earlier, until very recently, the literature on black-box parallel convex optimization was extremely scarce. Here we summarize the main lines of work.

\noindent{\bf Worst-case Lower Bounds for Sequential Convex Optimization.} 
    Classical theory of (sequential) oracle complexity in optimization was developed by Nemirovski and Yudin in~\cite{Nemirovski:1983}. This work provides sharp worst-case lower bounds for nonsmooth optimization, and a suboptimal (and rather technical) lower bound for randomized algorithms, for $\ell_p$ settings, where $1\leq p\leq \infty$. Smooth convex optimization in this work is addressed by lower bounding the oracle complexity of convex quadratic optimization, which only applies to deterministic algorithms and the $\ell_2$ setup. Nearly-tight lower bounds for deterministic non-Euclidean smooth convex optimization were obtained only recently \cite{Guzman:2015}, mostly by the use of a smoothing of hard nonsmooth families. It is worth mentioning that none of these lower bounds are robust to parallelization.
    
\noindent{\bf Lower Bounds for Parallel Convex Optimization.} 
    The study of parallel oracle complexity in convex optimization was initiated by Nemirovski in~\cite{Nemirovski:1994}, providing a worst-case lower bound $\Omega\big((\frac{d}{\ln(2Kd)})^{1/3}\ln(1/\varepsilon)\big)$ on the complexity in the $\ell_{\infty}$-setup. The argument from~\cite{Nemirovski:1994} is based on a sequential use of the probabilistic method to generate the subgradients of a hard instance and applies to an arbitrary dimension beyond a fixed constant. The author conjectured that this lower bound is suboptimal, which still remains an open problem.
    
    More recently, several lower bounds have been obtained for various settings of parallel convex optimization, but all applying only to \emph{either box ($\ell_\infty$-ball) or $\ell_2$-ball constrained Euclidean spaces}. In particular,~\cite{smith2017interaction} showed that poly-log in $1/\eps$ oracle complexity is not possible with polynomially-many in $d$ parallel queries for nonsmooth Lipschitz-continuous minimization. This bound was further improved by~\cite{duchi2018minimax}, in the context of stochastic minimization with either Lipschitz-continuous or smooth and strongly convex objectives. 
    
    Tight lower bounds in the Euclidean setup have been obtained in~\cite{nips2018-woodworth} and \cite{balkanski2018parallelization}. 
    Both of these works provide a tight lower bound $\Omega(1/\varepsilon^2)$ for randomized algorithms and nonsmooth Lipschitz objectives, when the dimension is sufficiently high (polynomial in $1/\varepsilon$, which is similar to our setting). The work in~\cite{balkanski2018parallelization} further considers strongly convex Lipschitz objectives. While this setting is not considered in our work, we note that it is possible to incorporate it in our framework using the ideas from~\cite{Sridharan:2012}. To obtain lower bounds that apply against randomized algorithms,~\cite{balkanski2018parallelization} uses an intricate adaptivity argument. Our lower bound is based on a more direct application of the probabilistic method, and is arguably simpler.
    
    The work in~\cite{nips2018-woodworth} further considers an extension to stochastic and smooth objectives. However, the ``statistical term'' in~\cite{nips2018-woodworth} comes from a typical minimax estimation bound, and its accuracy can, in fact, be reduced by parallelization at a rate $1/\sqrt{N}$, where $N$ is the total number of queries. Their construction of subgradients for the hard function is based on random vectors from the unit sphere; our use of Rademacher sequences makes the analysis simpler and more broadly applicable. On the other hand,~\cite{nips2018-woodworth} also provides lower bounds for (non-local) prox oracles, which are not considered in this paper.
    
\noindent{\bf Adaptive Data Analysis.} In a separate line of work, there has been significant progress in understanding adaptivity in data analysis, with the goal of preventing overfitting. The typical learning model used in this framework is the Statistical Query (SQ) model, which applies to stochastic convex optimization~\cite{Feldman:2017}. In this literature, it was proved that, given a dataset of size $n$, the number of adaptive SQs that can be accurately answered is $\tilde\Theta(n^2)$, and this is achieved by an application of results from differential privacy~\cite{Dwork:2015,Steinke:2015}. Improved results for low-variance SQs have been obtained in~\cite{Feldman:2018}. The power of nonadaptive SQ algorithms for PAC learning has been studied in~\cite{Daniely:2018}, and is characterized by a large margin condition. Negative results here do not translate to the context of convex optimization; however, as mentioned earlier, our negative results rule out specific approaches to improve the sample complexity of stochastic convex optimization.
%
%%%%%%%%%%%%%%%%%%%%%%%%%%%% PRELIMINARIES
\subsection{Notation and Preliminaries}\label{sec:prelims}
%
%
%
%
%%%%%%%%%%%% VECTOR SPACES
\paragraph*{Vector Spaces and Classes of Functions.} 
Let $(\vecspace,\|\cdot\|)$ be a $d$-dimensional normed vector space, where $d < \infty$. We will denote vectors in this space by bold letters, e.g., $\bx,\by$, etc.
We denote by $(\vecspace^{\ast},\|\cdot\|_{\ast})$ its dual space,
and we use the bracket notation $\innp{\bz, \bx}$ to denote the evaluation of the linear functional $\bz\in\vecspace^{\ast}$
at a point $\bx\in \vecspace$; in particular, $\|\bz\|_{\ast}=\sup_{\|x\|\leq 1}\innp{\bz,\bx}$. 
We denote the ball of $\vecspace$ centered at $\bx$ and of radius $r$ by ${\cal B}_{\|\cdot\|}(\bx,r)$, and the unit ball by ${\cal B}_{\|\cdot\|}:={\cal B}_{\|\cdot\|}(0,1)$. Our most important case of study is the space $\ell_p^d=(\RR^d,\|\cdot\|_p)$, where $1\leq p\leq \infty$. For simplicity, in this case we will use the notation ${\cal B}_p^d(\bx,r):={\cal B}_{\|\cdot\|_p}(\bx,r)$. The dual space of $\ell_p^d$ is isometrically isomorphic to $\ell_{p^*}^d$, where $p^*=p/(p-1)$; in this case, the bracket is just the standard inner product in $\RR^d$. Other important example is the case of Schatten spaces: $\mbox{Sch}_{p}^d=(\RR^{d\times d},\|\cdot\|_{Sch,p})$. Here, for any $\bX\in\RR^{d\times d}$, $\|\bX\|_{Sch,p}=(\sum_{i=1}^d\sigma_i(\bX)^p)^{1/p}$,
where $\sigma_1(\bX),\ldots,\sigma_d(\bX)$ are the singular values of $\bX$.

Given $\kappa\geq 0$, we use ${\cal F}_{(\vecspace,\|\cdot\|)}^{\kappa}(\mu)$ to denote the class of convex functions 
$f:\vecspace\rightarrow \RR$ such that 
\begin{equation} \label{eq:high_order_smooth}
 \opnorm{ D^{\lfloor \kappa+1\rfloor} f(\by) -D^{\lfloor \kappa+1\rfloor}(\bx)} \leq \mu \|\by-\bx\|^{\kappa+1-\lfloor\kappa+1\rfloor} \qquad(\forall \bx,\, \by\in \vecspace),
\end{equation}
where $D^{t}$ is the $t^\mathrm{th}$ derivative operator (a $t^{\mathrm{th}}$ order multilinear form), and $\opnorm{A}:=\sup_{\|\bh\|\leq 1}|A[\bh;\ldots;\bh]|$ is the induced 
operator norm on symmetric multilinear forms w.r.t. $\|\cdot\|$.

To clarify this definition, let 
us provide some useful examples:
\begin{itemize}
\itemsep0em 
\item $\kappa=0$ corresponds to 
bounded variation of subgradients, $\|\nabla f(\bx)-\nabla f(\by)\|_{\ast}\leq \mu$. 
This class contains all $\mu/2$-Lipschitz convex functions, but is also invariant under affine perturbations.\footnote{Our lower bounds for nonsmooth optimization are in fact given by classes of Lipschitz convex functions, but to keep the notation unified we use \eqref{eq:high_order_smooth} instead.}
\item $0<\kappa<1$ corresponds to H\"older continuity of the gradient, $\|\nabla f(\by)-\nabla f(\bx)\|_{\ast}\leq \mu \|\by-\bx\|^{\kappa}$.
\item $\kappa=1$ corresponds to Lipschitz-continuity of the gradient, $\|\nabla f(\by)-\nabla f(\bx)\|_{\ast}\leq \mu\|\by-\bx\|$.
\item $\kappa=2$ corresponds to Lipschitz-continuous Hessian, $\opnorm{ H f(\by)- H f(\bx)}\leq \mu\|\by-\bx\|$. 
If $\|\cdot\|$ comes from an inner product, then 
the operator norm is the largest singular value of the operator.
\end{itemize} 

%
%
%%%%%%%%%%%%%%%%% PROBLEMS, ALGOS, ORACLES
\paragraph*{Optimization Problems, Algorithms, and Oracles.}
We consider convex programs of the form
$$ \min\{ f(\bx):\,\, \bx\in \feasset  \},$$
where $f:\vecspace\to \RR$ is a convex function from a given class of objectives ${\cal F}$
(such as the ones described above), and $\feasset\subseteq \vecspace$ is convex and closed. We denote by $f^{\ast}$ 
the optimal value of the problem.\footnote{In general, to guarantee existence, it is required that $\feasset$ is compact. Our lower bound constructions will not require this assumption.} Our goal is, given an
accuracy parameter $\varepsilon>0$, to find an $\varepsilon$-solution; i.e., an $\bx\in\feasset$ such that
$f(\bx)-f^{\ast}\leq \varepsilon.$

We study complexity of convex optimization in the oracle model of computation. In this model, the algorithm queries points
from the feasible set $\feasset$, and it obtains partial information about the objective via a {\em local oracle}
${\cal O}$. 
Given objective $f\in{\cal F}$, and a query $\bx\in\feasset$, we denote the oracle answer by ${\cal O}_f(\bx)$
(when $f$ is clear from the context we omit it from the notation). 
We say that an oracle ${\cal O}$ is local if given two functions $f, g: \vecspace\rightarrow \mathbb{R}$ such that 
$f \equiv g$ in the neighborhood of some point $\bx \in \feasset,$ it must be that ${\cal{O}}_f(\bx) = {\cal{O}}_g(\bx)$. 
Notable examples of local oracles are the gradient over the class ${\cal F}_{\|\cdot\|}^{\kappa}(\mu)$, with $\kappa>0$;\footnote{When $\kappa=0$, not every subgradient oracle is local. However, this is a reasonable assumption for black-box algorithms (e.g, when we cannot access a dual formulation, or a smoothing of the objective).}
and a $\kappa^{\mathrm{th}}$-order Taylor expansion over the class ${\cal F}_{\|\cdot\|}^{\kappa}(\mu)$, with $\kappa$ being a non-negative
integer.

In the $K$-parallel setting of convex optimization~\cite{Nemirovski:1994}, an algorithm works in rounds. 
At every round, it performs a batch of queries 
$$ X^t=\{\bx_1^t,\ldots,\bx_K^t\},\quad \text{ for }\quad \bx_k^t\in \feasset\; (\forall k \in [K]),$$
where we have used the shorthand notation $k \in [K]$ to denote $k \in \{1,\dots, K\}.$

Given the queries, the local oracle ${\cal O}$ replies with a batch of answers:
$$ {\cal O}_f(X^t):= ({\cal O}_f(\bx_1^t),\ldots,{\cal O}_f(\bx_K^t)).
$$
The algorithm may work adaptively over rounds: every batch of queries may depend on 
queries and answers from previous rounds:
 \begin{equation}\label{eqn:updates}
 X^{t+1}= U^{t+1}(X^1,{\cal O}_f(X^1),\ldots,X^t,{\cal O}_f(X^t)) \qquad(\forall t\geq 1),
\end{equation}
where the first round of queries $X^1=U^1(\emptyset)$ is an instance-independent batch (the algorithm has no specific information about $f$ at the beginning). Functions $(U^t)_{t\geq 1}$, may be deterministic
or randomized, and this would characterize the deterministic or randomized nature of the algorithm. 
We are interested in the effect of parallelization on the complexity of convex optimization in the described oracle model. 
Notice that $K=1$ corresponds to the traditional notion of (sequential) oracle complexity.
%
%
%%%%%%%%%%%%%%%%% COMPLEXITY
\paragraph*{Notion of Complexity.}
Let ${\cal O}$ be a local oracle for a class of functions ${\cal F}$, and 
let ${\cal A}^K({\cal O})$ be the class of $K$-parallel deterministic algorithms interacting with oracle
${\cal O}$. Given $\varepsilon>0$, $f\in {\cal F}$, and $A\in{\cal A}^K({\cal O})$, define the running time
$T(A,f,\varepsilon)$ as the minimum number of rounds before algorithm $A$ finds an $\varepsilon$-solution. The notion of complexity used in this work is known as the \emph{high probability} complexity, defined as:
$$ \Compl_{\mbox{\footnotesize{HP}}}^{\gamma}({\cal F},\feasset,K,\varepsilon)
=  \sup_{F\in \Delta({\cal F})} \inf_{A\in{\cal A}^K({\cal O})}  \inf\{\tau:\, \prob_{f\sim F}[T(A,f,\varepsilon)\leq\tau]\geq \gamma\}, $$
where $\gamma \in (0, 1)$ is a confidence parameter and $\Delta({\cal F})$ is the set of probability distributions over the class of functions ${\cal F}$. 
The high probability complexity subsumes other well-known notions of complexity, including distributional, randomized, and worst-case, in the local oracle model. More details about the relationship between these different notions of complexity %in this model 
are provided in Appendix~\ref{app:complexity} and can also be found in~\cite{Braun:2017}.

\paragraph{Additional Background.} Additional background and statements of several useful definitions and facts that are important for our analysis are provided in Appendix~\ref{app:background}.
\subsection{Organization of the Paper} 
Next section provides a general lower bound that is the technical backbone of all the results in this paper. Section~\ref{sec:apps} then overviews the applications of this result in the general $\ell_p$ setups. Omitted proofs from Sections~\ref{sec:main-thm} and~\ref{sec:apps} are provided in Appendices~\ref{app:proofs-main-thm} and~\ref{app:proofs-apps}, respectively. We conclude in Section~\ref{sec:conclusion} with a discussion of obtained results and directions for future work. 
%
%%%%%%%%%%%%%%%%%%%%%%%%%%%% MAIN THEOREM
\section{General Complexity Bound}\label{sec:main-thm}
To prove the claimed complexity results from the introduction, we will work with a suitably chosen class of random nonsmooth Lipschitz-continuous problem instances. The results for the classes of problems with higher order of smoothness will be established (mostly) through the use of smoothing maps. In particular, we will make use of the following definition of locally smoothable spaces:
\begin{definition} \label{def:loc_smooth}
A space $(\vecspace,\|\cdot\|)$ is {\em $(\kappa,\eta,r,\mu)$-locally smoothable} if there
exists a mapping
$$ 
\begin{array}{rccl}
{\cal S}:& {\cal F}_{(\vecspace,\|\cdot\|)}^0(1) &\to &  {\cal F}_{(\vecspace,\|\cdot\|)}^{\kappa}(\mu)\\
	    & f					  &\mapsto & {\cal S}f
\end{array}\vspace{-10pt}
,$$
referred to as the local smoothing, such that: 
(i) $\|f-{\cal S}f \|_{\infty} \leq \eta$; and 
(ii) if $f,g\in {\cal F}_{\|\cdot\|}^0(1)$ and $\bx\in\vecspace$ are such that $f|_{B_{\|\cdot\|}(\bx,2r)}\equiv g|_{B_{\|\cdot\|}(\bx,2r)}$ then 
${\cal S}f|_{\mathcal{B}_{\|\cdot\|}(\bx,r)} \equiv {\cal S}g|_{\mathcal{B}_{\|\cdot\|}(\bx,r)}$. 
\end{definition}
Namely, a space is $(\kappa,\eta,r,\mu)$-locally smoothable if there exists a mapping ${\cal S}$ that maps all nonsmooth functions 
to functions in ${\cal F}_{\|\cdot\|}^{\kappa}(\mu)$, such that a function $f$ and its map ${\cal S}f$ do not differ by more than $\eta$ when evaluated at any point from the space, and the map preserves the equivalence of functions over sufficiently small neighborhoods of points from the space. This last property is crucial to argue about the behavior of a local oracle. 

The following theorem is the backbone of all the results from this paper: all complexity bounds will be obtained as its applications.
\begin{theorem} \label{thm:MetaThm}
Let $(\vecspace,\|\cdot\|)$ be a normed space and  
$\feasset \supseteq {\cal B}_{\|\cdot\|}$ be a closed and convex subset of $\vecspace$. 
Suppose there exist a 
positive integer $M$, independent random vectors $\bz^1,\ldots,\bz^M$ supported on ${\cal B}_{{\|\cdot\|}^{\ast}}$, $\varepsilon>0$, $\alpha>0$, and $0<\gamma<1/2$, such that, if we define
$\bar \delta=16\sqrt{\frac{\ln(MK/\gamma)}{\alpha}}$, we have: 
\begin{enumerate}
\itemsep0em 
\item[(a)] $(\vecspace,\|\cdot\|)$ is $(\kappa,\eta,r,\mu)$-locally smoothable, with $\mu>0$,
 $0< r \leq \bar\delta/8$, and $\eta \leq \varepsilon\mu/4$;
\item[(b)] 
$
\prob\big[\inf_{\bm{\lambda}\in \Delta_M}\big\|\sum_{i\in [M]}\lambda_i\bz^i\big\|_* \leq 4\mu\eps\big]\leq \gamma;
$
\item[(c)] 
For any $i\in[M]$, $\bx\in {\cal B}_{\|\cdot\|}$, and $\delta>0$
$$ \prob[ \innp{ \bz^i,\bx}  \geq \delta] \leq \exp\{-\alpha\delta^2 \} 
\quad\text{ and }\quad
 \prob[ \langle \bz^i,\bx\rangle \leq -\delta] \leq \exp\{-\alpha\delta^2 \}; $$
\item[(d)] $\bar\delta\leq \mu\varepsilon/M$.%
\end{enumerate} 
{Then, the high probability complexity of class ${\cal F}_{(\vecspace,\|\cdot\|)}^{\kappa}(1)$ on $\feasset$ satisfies
$$\mathrm{Compl}_{\mathrm{\footnotesize HP}}^{2\gamma}({\cal F}_{(\vecspace,\|\cdot\|)}^{\kappa}(1),\feasset,K,\varepsilon) \geq M.$$}
\end{theorem}

\begin{remark} \label{rem:rescaling_LB}
Theorem~\ref{thm:MetaThm} is stated for domains containing the unit ball and function class ${\cal F}_{(\vecspace,\|\cdot\|)}^{\kappa}(1)$. Handling arbitrary radius $R>0$ and regularity constant $\mu$ can be achieved by a simple rescaling and change of variables, which we omit for space considerations. The result is that if the lower bound for $R=\mu=1$ is $M(\varepsilon)$, then the lower bound for arbitrary $R, \mu>0$ would be $M(\sfrac{\varepsilon}{(\mu R^{\kappa+1})})$.
\end{remark} 
\begin{remark}
{Even though Theorem~\ref{thm:MetaThm} is stated for the standard setting, in which ${\cal X}$ contains the unit ball w.r.t.~the norm of the space, $\|\cdot\|$, it is possible to extend it in a generic way to non-standard settings in which these two norms do not agree. For an example of such a setting, see Theorem~\ref{thm:nonstandar-1p2}.}
\end{remark}
To prove Theorem~\ref{thm:MetaThm}, we
need to build a distribution over ${\cal F}_{(\vecspace,\|\cdot\|)}^{\kappa}(1)$ such that any $K$-parallel deterministic 
algorithm interacting with a local oracle on $\feasset$ needs $M$ rounds to reach an $\varepsilon$ solution, with probability $1-2\gamma$. We propose a family of objectives as follows.  
Given $\bz^1,\ldots,\bz^{M}$ as in the theorem,
consider the problem
(P) $\min\{F(\bx):
\bx\in \feasset\}$, where:%
\begin{equation} %\tag{$P$}
F(\bx):= \frac{1}{\mu}{\cal S}\Big(\max\Big\{ \frac12\max_{i\in[M]}\big[\innp{\bz^i,\cdot} - i\bar\delta\big],\; \|\cdot\| - \frac12(3(1+r) + M{\bar\delta}) \Big\}\Big)(\bx),
\label{eqn:rnd_pb}
\end{equation}
By construction, $F\in{\cal F}_{\|\cdot\|}^{\kappa}(1)$ surely. Observe that, since $\|\bz^i\|_*\leq 1$, for all $i$:
\begin{itemize}
\itemsep0em 
    \item[($O_1$)] When $\|\bx\|\leq 1+2r,$ it must be $1/2\max_{i\in[M]}\big[\innp{\bz^i,\bx} - i\bar\delta\big] \geq \|\bx\| -1/2( 3(1+r) + M{\bar\delta})$; i.e., within the unit ball, $F$ is only determined by its left term (and not the norm term).
    \item[($O_2$)] When $\|\bx\| \geq 3(1+r)+(M-1)\bar\delta$, it  must be $1/2\max_{i\in[M]}\big[\innp{\bz^i,\bx} - i\bar\delta\big] \leq \|\bx\| -1/2(3(1+r) - M{\bar\delta})$; i.e., outside the ball of radius $3(1+r)+(M-1)\bar\delta\leq 4,$\footnote{From (b) we may assume that $4\mu\varepsilon\leq 1$, and then using the bounds on $r$ and $\bar \delta$ from (a) and (d), we get the bound.} $F$ is only determined by the norm term (and not by $\bz^1,\ldots,\bz^M$).
\end{itemize}

We claim that any $K$-parallel deterministic algorithm that works in $M$ rounds, with probability $1-2\gamma$, will fail
to query a point with optimality gap less than $\varepsilon$.  
This suffices to prove the theorem. The proof consists of three main parts: (i) establishing an upper bound on the minimum value $F^\ast$ of~\eqref{eqn:rnd_pb}, which holds with probability $1-\gamma,$ (ii) establishing a lower bound on the value of the algorithm's output $\min\{ F(\bx): \bx \in \bigcup_{t\in[M]} X^t\},$ which holds with probability $1-\gamma$, and (iii) combining the first two parts to show that the optimality gap 
$\min\{F(\bx)-F^{\ast}:\bx\in\bigcup_{t\in[M]} X^t\}$ 
of the best solution found by the algorithm after $M$ rounds is higher than $\eps$, with probability $1-2\gamma$.  
The full proof is provided in Appendix~\ref{app:proofs-main-thm}.
%
%

%
%
%%%%%%%%%%%%%%%%%%%%%%%%%%%% APPLICATIONS
\section{Lower Bounds for Parallel Convex Optimization over $\ell_p$ Balls}\label{sec:apps}
In this section, we show how the general complexity bound from Theorem~\ref{thm:MetaThm} can be applied to obtain several lower bounds for parallel convex optimization. Our main case of study will be $\ell_p^d$ spaces. 
\begin{remark} \label{rem:Schatten_LB}
In what follows, we will prove several lower bounds for $\ell_p$-setups. Interestingly, we can obtain analog lower bounds for Schatten spaces. This can be obtained by simply noting that the restriction of the Schatten norm to diagonal matrices coincides with $\|\cdot\|_p$, and therefore we can embed ${\cal B}_p^d$, as well as ${\cal F}_{\ell_p^d}^{\kappa}(1)$ through this restriction (for more details, we refer the reader to \cite{Guzman:2015}). This embedding has a quadratic cost in the large-scale regime; in particular, it remains polynomial in $1/\varepsilon$ and $\ln(K/\gamma)$. 
\end{remark}
\subsection{Nonsmooth Optimization} \label{subsec:NonsmoothLB}
To apply Theorem~\ref{thm:MetaThm} in the nonsmooth case, we do not need to apply any smoothing at all. This is formally stated as ``any normed space is $(0,0,0,1)$-locally smoothable,'' and its consequence is that Property (a) of the theorem is automatically satisfied. 
Thus, it suffices to construct a probability distribution over $\bz^i$'s that under suitable constraints on $\alpha$ and the number of rounds $M$ satisfies Assumptions (b) and (c) from the theorem. Assumption (d) simply constrains $M$ by $M \leq \frac{\eps}{\bar \delta}.$

Let $\br^i$ denote an independent (over $i$) $d$-dimensional vector of independent Rademacher entries (i.e., a vector whose entries take values $\pm 1$ w.p.~$1/2,$ independently of each other). Let $\mI_L^i$ denote the $d \times d$ diagonal matrix, whose $L \leq d$ diagonal entries take value 1, while the remaining entries are zero. The positions of the non-zero entries on the diagonal of $\mI_L^i$ will, in general, depend on $i,$ and will be specified later. Given $p \geq 1$, 
vectors $\bz^i\in{\cal B}_{p^*}^d$ are then defined as:
\begin{equation}\label{eq:zi-def}
    \bz^i = \frac{1}{L^{1/p^*}}\mI^i_L \br^i. 
\end{equation}
\subsubsection{Bounds for $1 \leq p \leq 2$} 
When $p \in [1, 2],$ it suffices to choose $L = d,$ so that $\bz^i = d^{-1/p^*}\br^i$. We start by proving a lower bound that applies in the regime when $d = \Omega(\mathrm{poly}(\log(K/\gamma), 1/\varepsilon^{p^*}))$. Hence the bound deteriorates as $p$ tends to one, and, in particular, does not apply to the case when $p=1.$ However, we will also show that it is possible to derive a lower bound for a restricted feasible set: the lower bound will apply to Lipschitz-continuous nonsmooth minimization over an $\ell_2$ ball inscribed in the unit $\ell_p$ ball and it will apply in the regime of $d = \Omega(\mathrm{poly}(\log(K/\gamma), 1/\varepsilon))$. This provides a strong indication that obtaining speedups from parallelizing convex optimization is not any easier when $p$ is close to 1 than in other regimes of $p$.   
The following lemma gives a sufficient condition for assumption (b) 
from Theorem~\ref{thm:MetaThm} to hold. Its proof is provided in Appendix~\ref{app:proofs-apps}.
\begin{restatable}{lemma}{lemoneptwoub}\label{lemma:1<p<2-ub}
Let $1< p \leq 2$ and let $\bz^1,\ldots,\bz^M$ be chosen according to Eq.~\eqref{eq:zi-def}, where 
$$ 
M \leq \min\Big\{\frac{1}{200\eps^2},  \frac{d/12 - \ln(1/\gamma)}{\ln(3/\eps)}\Big\},$$ 
then for all $\gamma \in (1/\mathrm{poly}(d), 1):$ $\prob[\min_{\bm{\lambda}\in \Delta_M}\|\sum_{i \in [M]}\lambda_i \bz^i\|_{p^*} \leq 4\eps] \leq \gamma.$
\end{restatable}

To obtain the claimed lower bound for the nonsmooth case, we only need to establish the concentration of inner products within the feasible domain. When $p>1$, this is obtained as a simple application of Hoeffding's Inequality. These two facts provide the claimed lower bound.
\begin{theorem}\label{thm:nonsmooth-1p2}
Let $1< p\leq 2$ and $\feasset \supseteq {\cal B}_p^d$. Let $\eps\in(0,1/2)$ and $\gamma\in(1/\mathrm{poly}(d),1)$. 
Then:
$$
\mathrm{Compl}_{\mathrm{HP}}^{\gamma}({\cal F}_{\ell_p^d}^{0}(1),\feasset,K,\varepsilon)\geq M := \min \Big\{ \frac{1}{200 \eps^2},\; \frac{\varepsilon d^{1/p^*}}{32\sqrt{\ln(MK/\gamma)}}
\Big\}.
$$
\end{theorem}
\begin{proof}
We verify the conditions of Theorem~\ref{thm:MetaThm}. Recall that in the nonsmooth case condition (a) is automatically satisfied. For (b), by a direct application of  Hoeffding's Inequality, for all $x\in {\cal B}_p^d$
$$ \prob[\langle \bz^i,\bx\rangle >\delta] = \prob[\langle \br^i,\bx\rangle >d^{1/p^*}\delta] \leq \exp\{-d^{2/p^*}\delta^2\}. $$
In particular, we have that $\alpha=d^{2/p^*}$ suffices to satisfy (b). Property (c) is obtained from Lemma~\ref{lemma:1<p<2-ub}, which requires bounding $M$ according to the lemma.  
Assumption (d) holds as long as $M \leq \eps/{\bar\delta}.$ As ${\bar\delta} = 16\sqrt{\frac{\ln(MK/\gamma)}{\alpha}}$, it is sufficient to require: 
$M\leq \frac{\varepsilon d^{1/p^*}}{32}\frac{1}{\sqrt{\ln(MK/\gamma)}}$.
\end{proof}
\begin{remark}
Even though $M$ is implicitly defined in Theorem~\ref{thm:nonsmooth-1p2}, an explicit definition for $M$ can be obtained by using a looser bound $\ln(dK/\gamma)$ instead of $\ln(MK/\gamma).$ We keep this definition to highlight the large scale regime for $d$. 
{In particular, the high-dimensional regime is determined by solving for $d$ the inequality $\frac{\varepsilon d^{1/p^*}}{32\sqrt{\ln(MK/\gamma)}} \geq M,$ where $M = \frac{1}{200 \varepsilon^2}$.}
\end{remark}
We can conclude from Theorem~\ref{thm:nonsmooth-1p2} that as long as $d$ is ``sufficiently large'' (namely, as long as 
$d=\Omega(\big(\sqrt{\ln(K/(\varepsilon\gamma))}/\varepsilon^3\big)^{p^*})$), any $\eps$-approximate $K$-parallel algorithm takes $\Omega(1/\eps^2)$ iterations, which is asymptotically optimal -- this bound is tight in the sequential case (when $K=1$) and is thus unimprovable~\cite{Nemirovski:1983}. Unfortunately, this lower bound becomes uninformative when $p^*=\Omega(\ln d)$; in particular, when $p=1$. 
\paragraph{A Lower Bound for a Nonstandard Setting.} As we mention above, none of the techniques of this paper is able to provide a $\Omega(1/\varepsilon^2)$ lower bound for the nonsmooth $\ell_1$-Lipschitz optimization over a unit $\ell_1$ ball. However, we can show a slightly weaker result: Namely, that $\ell_1$-Lipschitz convex optimization over a subset of the ${\cal B}_1^d$-ball has parallel complexity $\Omega(1/\varepsilon^2)$. In fact, this result holds more generally for $\ell_p$-Lipschitz convex optimization, where $p \in [1, 2],$\footnote{When $p=2$, the inscribed $\ell_2$ ball is exactly the unit $\ell_p$ ball.} over an $\ell_2$ ball inscribed in the unit $\ell_p$ ball. The proof is provided in Appendix~\ref{app:proofs-apps}.  

\begin{restatable}{theorem}{nonstanoneptwo}\label{thm:nonstandar-1p2}
 Let $\varepsilon\in(0,1/2)$, $\gamma\in (1/\mathrm{poly}(d),1)$, and $p \in [1, 2]$. Then:
$$
\mathrm{Compl}_{\mathrm{HP}}^{\gamma}({\cal F}_{\ell_p^d}^{0}(1),{\cal B}_2^d(1/{d}^{1/p-1/2}),K,\varepsilon)\geq M := \min \Big\{ \frac{1}{200 \eps^2},\; \frac{\varepsilon d^{1/2}}{32\sqrt{\ln(MK/\gamma)}}
\Big\}.
$$
\end{restatable}
\subsubsection{Bounds for $p \geq 2$}\label{sec:p>2}
It is possible to extend Lemma~\ref{lemma:1<p<2-ub} to the case of $p\geq2.$ However, due to the upper bound on $M$ from Lemma~\ref{lemma:1<p<2-ub}, the best dimension-independent lower bound on the number of queries we could obtain in this setting would be of the order $1/\eps^2.$ Given that in the sequential setting the best dimension-independent lower bound is $\Omega(1/\eps^p),$ we need a stronger result than what we obtained in Lemma~\ref{lemma:1<p<2-ub}. 

This is achieved through a different construction of  $\bz^i$'s, where these vectors are no longer supported on all $d$ coordinates, but only on $L<d$ of them; moreover, we will choose their supports to be disjoint. The construction is as follows. Let $\{J_i\}_{i=1}^M$ be a collection of subsets of $\{1,\dots, d\}$ such that $|J_i| = L$  and $J_i \cap J_{i'} = \emptyset,$ $\forall i \neq i'$ (here, we assume that $d \geq ML$). Set $\mI_L^i = \mathrm{diag}(\mathds{1}_{J_i}),$ i.e., the $(j, j)$ element of the diagonal matrix $\mI_L^i$ is 1 if $j \in J_i$ and 0 otherwise. As before (see~\eqref{eq:zi-def}), $\bz^i$ is defined as $\bz^i = \frac{1}{L^{1/p^*}}\mI^i_L \br^i,$ where $(r_j^i)_{i\in[M],j\in [d]}$ is an independent Rademacher sequence. 

Our next result addresses the nonsmooth $p\geq 2$ case, by a direct application of Theorem~\ref{thm:MetaThm} to our construction above. More details are provided in Appendix~\ref{app:proofs-apps}.
\begin{restatable}{theorem}{ptwothmnsmooth}\label{thm:p>2-nonsmooth}
Let $p\geq 2$, $\feasset \supseteq {\cal B}_p^d$, and $\eps \in (0, \sfrac{1}{2}),$ $\gamma \in (\sfrac{1}{\mathrm{poly}(d)}, 1).$ Then:  
$$
\mathrm{Compl}_{\mathrm{HP}}^{\gamma}({\cal F}_{\ell_p^d}^{\kappa}(1),\feasset,K,\varepsilon)\geq M := \min \Big\{ \frac{1}{(4\eps)^p},\; \frac{\eps^{2/3}}{8}\Big(\frac{d}{\ln(MK/\gamma)}\Big)^{1/3} \Big\}. 
$$
\end{restatable}
In particular we have that the required number of queries to reach an $\eps$-approximate solution is $\Omega(\frac{1}{\eps^p})$, as long as $d = \Omega(\frac{\ln(K/\gamma) + p\ln(1/\eps)}{\eps^{3p+2}}).$ When $p \rightarrow \infty,$ the right term in the definition of $M$ dominates, and we have $M = \Omega\Big(\eps^{2/3}\big(\frac{d}{\ln(dK/\gamma)}\big)^{1/3}\Big),$ which, for constant $\eps,$ matches the best known bound for deterministic algorithms in this setting, due to~\cite{Nemirovski:1994}. 
\subsection{Smooth and Weakly Smooth Optimization}
\label{subsec:SmoothLB}
To apply Theorem~\ref{thm:MetaThm} and obtain lower bounds for (weakly) smooth classes of functions, we need to 
design an appropriate 
local smoothing. %
This is indeed possible for $p \geq 2,$ as we show below.

\begin{remark} \label{rem:smoothings}
Here we list some known local smoothings from the literature.
\begin{enumerate}
\itemsep0em 
\item Let $2\leq p\leq \infty$, $d\in\mathbb{N}$, and $0\leq \kappa\leq 1$. Then, for any $\eta>0$, 
the space $\ell_p^d=(\mathbb{R}^d,\|\cdot\|_p)$ is $(\kappa,\eta,\eta,\mu)$-locally
smoothable when
$\mu=2^{1-\kappa}(\min\{ p,\ln d\}/\eta)^{\kappa}$. We prove this in the Appendix~\ref{app:background}, following \cite{Guzman:2015}.
\item Let $d,\kappa\in\mathbb{N}$ and $\eta>0$. Then $\ell_2^d$ is 
$(\kappa,\kappa\eta,\kappa\eta,(d/\eta)^{\kappa})$-locally smoothable. This is achieved by a sequential integral convolution w.r.t.~the
uniform kernel on the ball of radius $\eta$~\cite{Agarwal:2018}. They also show that for $1\leq L\leq d$, the restriction of ${\cal S}$ to the set:
$$ 
\big\{ f:\RR^d\to \RR:\,\, f\in{\cal F}_{\ell_2^d}^{0}(1),\,\, (\exists\, \Gamma \mbox{ subspace of dim. }L)\,
(\forall \by\in \Gamma^{\perp})\,\, f(\bx)=f(\bx+\by) \big\},
$$  
satisfies an improved 
$(\kappa,\kappa\eta,\kappa\eta,(L/\eta)^{\kappa})$ local smoothing
property. 
\end{enumerate}
\end{remark}

Our next result addresses the smooth $\ell_p^d$-setup when $p\geq 2$. Its proof is provided in Appendix~\ref{app:proofs-apps}.
\begin{restatable}{theorem}{ptwothmsmooth}\label{thm:p>2-smooth}
Let $p\geq 2$, $\feasset \supseteq {\cal B}_p^d$, and $\eps \in (0, 1/2),$ $\gamma \in (1/\mathrm{poly}(d), 1).$ Then:  
\begin{align*}
   \mathrm{Compl}_{\mathrm{HP}}^{\gamma}({\cal F}_{\ell_p^d}^{\kappa}(1),\feasset,K,\varepsilon)\geq M := &\min \bigg\{ \Big(\frac{1}{2^{3+4\kappa}\, \eps\, (\min\{p, \ln(d)\})^\kappa}\Big)^{\frac{p}{1+\kappa(1+p)}}, \;  \\
    &\frac{d}{2^9\ln(MK/\gamma)}\left(2^{\frac{1+3p + 2\kappa(1+p)}{1+p}}\min\{p, \ln(d)\}^\kappa\eps\right)^{\frac{2(1+p)}{1+\kappa(1+p)}}\bigg\}.
\end{align*}
\end{restatable}
\noindent The bound from Theorem~\ref{thm:p>2-smooth} may be difficult to read, so let us point out a few notable special cases:
\begin{itemize}
\itemsep0em 
    \item When $\kappa = 0,$ $p \rightarrow \infty,$ the bound is uninformative, and one should instead use Theorem~\ref{thm:p>2-nonsmooth}. This is a consequence of the particular choice of $L$ in the proof, and its dependence on $\kappa$. 
    \item When $\kappa \in (0, 1],$ $p \rightarrow \infty,$ if $d = \Omega\left((\ln(\frac{K}{\gamma}) + \frac{1}{\kappa}\ln(\frac{1}{\eps}))(\frac{1}{\eps})^{\frac{3}{\kappa}}\right),$ then $M = \frac{1}{\ln(d)}(\frac{1}{2^{3+4\kappa}\eps})^{1/\kappa},$ which is tight up to a factor $\frac{1}{\ln(d)}$ and achieved for $K=1$ by~\cite{Frank:1956} method.
    \item When $\kappa = 0,$ $p < \infty,$ and $d = \Omega \left((\ln(K/\gamma) + p\ln(1/\eps))(\frac{1}{\eps})^{3p+2} \right),$ then $M = (\frac{1}{8\eps})^p,$ which is achieved for $K=1$ by the Mirror-Descent method~\cite{Nemirovski:1983}.%, 
    \item When $\kappa = 1,$ $p < \infty,$ and $d = \Omega \left(\max\{(\ln(K/\gamma) + \ln(1/\eps))(\frac{1}{\eps})^{3}, \exp(p)\} \right),$ then $M=(\frac{1}{128 p\eps})^{\frac{p}{p+2}}$. %
    These bounds are unimprovable and are achieved for $K=1$ by the Nemirovski-Nesterov accelerated method~\cite{Nemirovski:1985,dAspremont:2018}.%
\end{itemize}

\begin{remark}
The proof strategy of Theorem~\ref{thm:p>2-smooth} for $p=2$ can also be used to obtain lower bounds for higher-order smooth convex optimization, following \cite{Agarwal:2018}. Namely, using the sequential integral convolution smoothing from Remark~\ref{rem:smoothings}, we can obtain analog lower bounds as in~\cite{Agarwal:2018}, that also apply to parallel algorithms. We defer the details of this simple corollary to the full version of the paper. 
\end{remark}

Unfortunately, the smoothing approach is not immediately applicable when $1\leq p<2$, due to
the fact that there are no known regularizers for an infimal convolution smoothing.  This is related to the fact that these spaces are not $2$-uniformly smooth (see, e.g., \cite{Ball:1994}) which leads to a natural barrier for the approach. 
However, this difficulty has been circumvented by \cite{Guzman:2015}, where 
lower bounds in this regime are shown by a reduction from the $p=\infty$ case, specifically through a linear embedding of problem classes. We follow the same approach, and for the sake of brevity, we only provide a proof sketch in Appendix~\ref{app:proofs-apps}.
\begin{restatable}{theorem}{oneptwothmsmooth}\label{thm:1<p<2-smooth}
Let $1\leq p<2$, $0<\kappa\leq 1$, $\feasset \supseteq {\cal B}_p^d$, $\eps \in (0, 1/2),$ $\gamma \in (1/\mathrm{poly}(d), 1).$ Then, there exist constants $\nu, c(\kappa)>0$, such that if $d \geq \frac{1}{\nu}\left\lceil2(\ln(\nu d K/\gamma))^{\frac{2\kappa}{3+2\kappa}}\big(\frac{1}{\eps}\big)^{\frac{6}{3+2\kappa}}\right\rceil$, then: % 
\begin{align*}
    \mathrm{Compl}_{\mathrm{HP}}^{\gamma}({\cal F}_{\ell_p^d}^{\kappa}(1),\feasset,K,\varepsilon)\geq M := \frac{c_\kappa}{\ln(1/\eps) + \kappa \ln\ln(dK/\gamma)}\Big(\frac{1}{\eps}\Big)^{\frac{2}{3+2\kappa}}.
\end{align*}
\end{restatable}
Let us consider some special cases of the bound from Theorem~\ref{thm:1<p<2-smooth}. Suppose that $d$ is sufficiently high-dimensional so that the theorem applies (note that $d = \Omega(\ln(dK/\gamma)\eps^{-2})$ suffices). When $\kappa = 1,$ then $M = \Omega(\frac{1}{\ln(1/\eps) + \ln\ln(dK/\gamma)}(\frac{1}{\eps})^{2/5}).$ This bound does not match the sequential complexity $\Theta(1/\sqrt{\eps})$ of this problem -- apart from the logarithmic factors, the exponent in $1/\eps$ is off by $1/10$. This is a direct consequence of the right term in Theorem~\ref{thm:p>2-smooth} not being large enough for $p \to \infty,$ as the bound in Theorem~\ref{thm:1<p<2-smooth} is obtained from this case. Further improvements of this term would also improve the bound for the nonsmooth $\ell_\infty$ case of~\cite{Nemirovski:1994} for, at least, some regimes of $\eps$. Similarly, when $\kappa = 0,$ the exponent in $1/\eps$ is $2/3$, which is off by additive $4/3$ from the sequential complexity of this setting.  
{This is aligned with the intuition that smooth lower bounds have a milder high-dimensional regime than nonsmooth ones (which holds in the sequential case). This way, the embedding approach is stronger on higher levels of smoothness.}

The main difficulty in obtaining tighter bounds in these regimes ($\ell_\infty$ and its implications on smooth and weakly-smooth $p \in [1, 2)$ settings) is in relaxing Assumption (d) from Theorem~\ref{thm:MetaThm}. It seems unlikely that this would be possible without completely changing the hard instance used in its proof (as Assumption (d) is crucially used in bounding below the optimality gap), and would likely require a fundamentally different approach from the one used here, as well as in the related work.
%
%
%%%%%%%%%%%%%%%%%%%%%%%%%%%% CONCLUSION
\section{Conclusion}\label{sec:conclusion}
This paper rules out the possibility of significantly improving the complexity of convex optimization via parallelization in the exploration of the feasible set with polynomially-bounded in the dimension number of queries per round, 
for essentially all interesting geometries and classes of functions with different levels of smoothness. Most of the obtained lower bounds match the sequential complexity of these problems, up to, at most, a logarithmic factor in the dimension, and are, thus, (nearly) tight.

However, our bounds only apply to the high-dimensional setting, where $d = \Omega (1/\mathrm{poly}(\eps)).$ In the low-dimensional setting, the only bound we are aware of is in terms of worst-case complexity (for deterministic algorithms) for nonsmooth optimization over the $\ell_\infty$ ball, due to~\cite{Nemirovski:1994}. The bound is $\Omega((\frac{d}{\ln(dK)})^{1/3}\ln(1/\eps)).$ 
It was conjectured in~\cite{Nemirovski:1994} that the correct bound for nonsmooth optimization over the $\ell_\infty$ ball should be $\Omega(\frac{d}{\ln(K)}\ln(1/\eps)).$ Our analysis recovers a bound similar to Nemirovski's result in the stronger high probability complexity model, but \emph{only for constant} $\eps.$ We conjecture that in the low-dimensional setting of both (weakly-)smooth and nonsmooth  optimization the correct answer should be  $\Omega(\frac{d}{\ln(K/\gamma)}\ln(1/\eps)).$ 

%
%
%%%%%%%%%%%%%%%%%%%%%%%%%%%% CONCLUSION
%
%
%%%%%%%%%%%%%%%%%%%%%%%%%%% ACKS

%
%
\section*{Acknowledgements}{
%Part of this work was done while JD was a Microsoft Research Fellow at the Simons Institute for the Theory of Computing, for the program on Foundations of Data Science, and while she was a postdoctoral researcher at Boston University. 
%JD was partially supported by the NSF grant \#CCF-1740855. CG was partially supported by the FONDECYT project 11160939, and the Millennium
%Science Initiative of the Ministry
%of Economy, Development, and
%Tourism, grant ``Millennium Nucleus Center for the
%Discovery of Structures in Complex Data.''
%
We thank Adam Smith for pointing out the importance of lower bounds in parallel convex optimization from the local differential privacy perspective, for his useful comments and insights that led to this work, and for many useful discussions. 
We also thank Ilias Diakonikolas for sharing his expertise on anticoncentration, as well as Yossi Arjevani, Nicolas Casabianca, Vitaly Feldman, and Jos\'e Verschae for valuable discussions and comments regarding this work. 
An earlier version of this paper made an incorrect use of the orthogonal splittings of $\ell_p$ spaces, which has since been corrected. We thank Boris Kashin for pointing out this issue. We also thank an anonymous COLT reviewer who {pointed out imprecise conditioning in the proof of Lemma~\ref{lemma:gen-lb-event}, which has since been corrected.}
}
\bibliographystyle{abbrv}
\bibliography{bibliography}

%  %%%%%%%%%%%%%%%%%%%%%%%%%% APPENDIX
\appendix 
%
%
%%%%%%%%%%%%%%%%%%%%%%%%%% ADDITIONAL BACKGROUND
\section{Additional Background}\label{app:background}
For completeness, this section provides additional background and statements of some known facts that are used in the proofs of our lower bounds.
%
%
%%%%%%%%%%%%%%%%%%%%%%%%%%% NOTIONS OF COMPLEXITY
\subsection{Notions of Complexity in the Local Oracle Model}\label{app:complexity}
The {\em worst-case} oracle complexity is defined as:
$$ \Compl_{\mbox{\footnotesize{WC}}}({\cal F},\feasset,K,\varepsilon)
= \inf_{A\in{\cal A}^K({\cal O})} \sup_{f\in {\cal F}} T(A,f,\varepsilon). $$
For the case of randomized algorithms, it can be shown \cite{Nemirovski:1983}
that their complexity is equivalent to the one obtained from the expected running time over mixtures of
deterministic algorithms. That means that we can define the {\em randomized} oracle complexity as:
$$ \Compl_{\mbox{\footnotesize{R}}}({\cal F},\feasset,K,\varepsilon)
= \inf_{R\in\Delta({\cal A}^K({\cal O}))} \sup_{f\in {\cal F}}  \mathbb{E}_{A\sim R}[T(A,f,\varepsilon)], $$
where $\Delta({\cal B})$ is the set of probability distributions on the set ${\cal B}$.

We may consider an even weaker notion of {\em distributional} oracle
complexity, defined as
$$ \Compl_{\mbox{\footnotesize{D}}}({\cal F},\feasset,K,\varepsilon)
=  \sup_{F\in \Delta({\cal F})} \inf_{A\in{\cal A}^K({\cal O})}  \mathbb{E}_{f\sim F}[T(A,f,\varepsilon)]. $$
In this case, it is important to note that lower bounds cannot be obtained from adversarial choices of
$f$, as the probability distribution on instances $F$ must be set before the algorithm is chosen. 
It is easily seen that:  
$$ \Compl_{\mbox{\footnotesize{D}}}({\cal F},\feasset,K,\varepsilon) \leq
\Compl_{\mbox{\footnotesize{R}}}({\cal F},\feasset,K,\varepsilon) \leq
\Compl_{\mbox{\footnotesize{WC}}}({\cal F},\feasset,K,\varepsilon).$$

Finally, given a confidence parameter $0<\gamma<1$, \emph{high probability} complexity is defined as:
$$ \Compl_{\mbox{\footnotesize{HP}}}^{\gamma}({\cal F},\feasset,K,\varepsilon)
=  \sup_{F\in \Delta({\cal F})} \inf_{A\in{\cal A}^K({\cal O})}  \inf\{\tau:\, \prob_{f\sim F}[T(A,f,\varepsilon)\leq\tau]\geq \gamma\}. $$
Notice that a lower bound on the high probability complexity with confidence parameter $\gamma$
gives a lower bound on the distributional complexity, by the law of total probability
$$  \Compl_{\mbox{\footnotesize{D}}}({\cal F},\feasset,K,\varepsilon)
\geq (1-\gamma)  \Compl_{\mbox{\footnotesize{HP}}}^{\gamma}({\cal F},\feasset,K,\varepsilon).$$
All lower bounds in this work are for high probability complexity, with $\gamma=1/\mbox{poly}(d)$.%,
%
%
%%%%%%%%%%%%%%%%%%%%%%%%%%%% SUPPORTING CLAIMS
\subsection{Geometry of $\ell_p$ Spaces}

In the proof of Theorem~\ref{thm:1<p<2-smooth}, we make use Dvoretzky's Theorem, on the existence of nearly Euclidean
sections of the $\|\cdot\|_p$ ball. Its full description and proof may be 
found in \cite[Theorem 4.15]{Pisier:1989}. Here we state a concise version with 
what is needed for our results.

\begin{theorem}[Dvoretzky]
There exists a universal constant $0<\alpha<1$ such that for any $d>1$, there 
exists a subspace $F\subseteq \RR^d$ of dimension at most $\alpha d$ and an
ellipsoid ${\cal E}\subseteq F$ such that
$$ \frac12{\cal E} \subseteq {\cal B}_p^d\cap F \subseteq {\cal E}.$$
\end{theorem}

\subsection{Smoothings}

\begin{claim}\label{claim:smoothing}
Let $2\leq p\leq \infty$, $d\in\mathbb{N}$, and $0\leq \kappa\leq 1$. Then, for any $\eta>0$, 
the space $\ell_p^d=(\mathbb{R}^d,\|\cdot\|_p)$ is $(\kappa,\eta,\eta,\mu)$-locally
smoothable when
$\mu=2^{1-\kappa}(\min\{ p,\ln d\}/\eta)^{\kappa}$. 
\end{claim}
\begin{proof} 
First, we use the fact from \cite[Proposition 1]{Guzman:2015} that $\ell_p^d$ is $(1,\eta,\eta,\mu)$-locally smoothable with parameter $\tilde\mu=\min\{p,\ln d\}/\eta$. This can be achieved by 
infimal convolution smoothing
$${\cal S}f(\bx) = \inf_{h\in {\cal B}_p(0,\eta)}[f(\bx+\bh)+\phi(\bh)] \qquad (\forall \bx\in \RR^d),$$
where $\phi(\bx)=2\|\bx\|_r^2$ with $r=\min\{p,3\ln d\}$ as a regularizer.
 Furthermore, in this reference it is proved that if $f$ is a 1-Lipschitz function, then not only  ${\cal S}f\in{\cal F}_{\ell_p^d}^1(\mu)$ but also ${\cal S}f$ is 1-Lipschitz as well; therefore, the following two inequalities hold for any $\bx, \by\in \RR^d$
\begin{eqnarray*}
\|\nabla f(\bx)-\nabla f(\by)\|_{\ast}^{1-\kappa}&\leq& 2^{1-\kappa}\\
\|\nabla f(\bx)-\nabla f(\by)\|_{\ast}^{\kappa}&\leq& \tilde\mu^{\kappa},
\end{eqnarray*}
and multiplying these inequalities, we obtain
$\|\nabla f(\bx)-\nabla f(\by)\|_{\ast}\leq 2^{1-\kappa}\tilde\mu^{\kappa}=\mu$.
\end{proof}
\subsection{Deviation Bounds}
Here we state some specific probabilistic deviation bounds that we need for our results.
The first one is the left-sided Bernstein inequality, which may be found in \cite[Chapter 2]{wainwright_2019}. 

\begin{theorem}[Left-Sided Bernstein Inequality] \label{thm:LeftBernstein}
Let $Y_1,\ldots, Y_n$ be nonnegative independent random variables, with finite second moment. 
Then, for any $\delta>0$,
$$\prob \Big[ \sum_{k=1}^n (Y_k-\EE[Y_k]) \leq -n\delta\Big]
\leq \exp\Big\{-\frac{n\delta^2}{\frac2n\sum_{k=1}^n \EE[Y_k^2]} \Big\}.$$
\end{theorem}

We also remind the reader of the Khintchine inequality, which provides bounds for $L^p$ moments of Rademacher sequences (see, e.g., \cite{haagerup1981best}).

\begin{theorem}[Khintchine]
 Let $0<p<\infty$. There exist constants $c_p,c_p^{\prime}>0$ 
 such that for any $x_1,\ldots,x_L\in \RR$, and $r_1,\ldots,r_L$ a Rademacher sequence
 $$ c_p \|\bx\|_2 \leq \Big(\EE\Big|\sum_{i=1}^L r_ix_i\Big|^p\Big)^{1/p}
 \leq  c_p^{\prime} \|\bx\|_2.$$
\end{theorem}

\subsection{Packings and Cardinality of $\varepsilon$-Nets}
To show that it is possible to satisfy the assumption of Lemma~\ref{lemma:meta-ub} in the proof of Theorem~\ref{thm:MetaThm}, we will frequently rely on the following simple lemma, which follows by constructing an $(\eps/M)$-net w.r.t.~$\ell_{\infty}^M$ of the simplex, $\Delta_M$.  
\begin{restatable}{lemma}{epsnet}\label{lemma:eps-net-for-ub}
If, $\forall \bm{\lambda} \in \Delta_M,$ $\prob\big[\big\| \sum_{i=1}^M \lambda_i \bz^i \big\|_* \leq (c+1)\eps \big]\leq \gamma'$ for $\eps \in (0, 1)$, $c > 0,$ and $\gamma' \in (0, 1),$ then: 
$$
    \prob\big[\min_{\bm{\lambda}\in \Delta_M}\big\| \sum_{i=1}^M \lambda_i \bz^i \big\|_* \leq c\eps \big]\leq \Big(\frac{3}{\eps}\Big)^M \gamma'.
$$
\end{restatable}
\begin{proof}
The proof follows by constructing an $(\eps/M)$-net $\Gamma$ w.r.t.~the $\ell_\infty$ norm. In particular, let $\Gamma$ be a discrete set of points from $\Delta_M$. To apply the argument, we need to establish that:
\begin{equation}\label{eq:eps-net-distance}
\left| \inf_{\bm{\lambda}\in\Delta_M}\Big\| \sum_{i=1}^M \lambda_i \bz^i \Big\|_* - \inf_{\bm{\lambda}'\in \Gamma}\Big\| \sum_{i=1}^M \lambda_i' \bz^i \Big\|_* \right| \leq \eps.
\end{equation}
For~\eqref{eq:eps-net-distance} to hold, it suffices to show that for every $\bm{\lambda} \in \Delta_M,$ there exists $\bm{\lambda}' \in \Gamma$ such that 
$$
\Big\| \sum_{i=1}^M \lambda_i' \bz^i \Big\|_* \leq \Big\| \sum_{i=1}^M \lambda_i \bz^i \Big\|_* + \eps.
$$
By the triangle inequality, 
\begin{align*}
    \big\| \sum_{i=1}^M \lambda_i' \bz^i \big\|_* - \big\| \sum_{i=1}^M \lambda_i \bz^i \big\|_* 
    &\leq \big\| \sum_{i=1}^M (\lambda_i' - \lambda_i)\bz^i \big\|_* \\
    &\leq M\|\bm{\lambda} - \bm{\lambda}'\|_\infty \max_{i \in [M]}\|\bz^i\|_*\\
    &\leq M\|\bm{\lambda} - \bm{\lambda}'\|_\infty,
\end{align*}
as $\bz^i \in {\cal B}_{\|\cdot\|_*}.$ Hence, it suffices to have $\|\bm{\lambda} - \bm{\lambda}'\|_\infty \leq \eps/M.$

Define the discrete set ($(\eps/M)$-net) $\Gamma$ to be the set of vectors $\bm{\lambda}'$ such that $\forall j \in \{1, ..., M\}:$ $\lambda_j' = n_j \left\lceil\frac{M}{\eps}\right\rceil^{-1},$ where $n_j \geq 0,$ $\forall M,$ and $\sum_{j=1}^M n_j = \left\lceil\frac{M}{\eps}\right\rceil.$ 
Clearly, for any $\bm{\lambda} \in \Delta_M,$ we can choose $\bm{\lambda}' \in \Gamma$ such that $\|\bm{\lambda} - \bm{\lambda}'\|_\infty \leq \eps/M.$ Applying the union bound  over $\bm{\lambda}'\in \Gamma$ and using the lemma assumption:
\begin{equation}%\label{eq:inf-over-eps-net}
\prob\bigg[ \inf_{\bm{\lambda}' \in \Gamma}\Big\| \sum_{i=1}^M\lambda_i' \bz^i \Big\|_{*} \leq (c+1)\varepsilon  \bigg]
\leq |\Gamma|\gamma'.\notag
\end{equation}
The size of the $\eps$-net $\Gamma$ can be bounded by $|\Gamma| = \binom{\left\lceil\frac{M}{\eps}\right\rceil+M}{M} \leq \left(\frac{3}{\eps}\right)^M$ using the standard stars and bars combinatorial argument. To complete the proof, it remains to apply the bound from Eq.~\eqref{eq:eps-net-distance}. 
\end{proof}
%
%
%
%%%%%%%%%%%%%%%%%%%%%%%%%%%% ANTI-CONCENTRATION
%
%
%%%%%%%%%%%%%%%%%%%%%%%%%%% OMITTED PROOFS FROM GENERAL
\section{Proof of Theorem~\ref{thm:MetaThm}}
\label{app:proofs-main-thm}
% 
%
%%%%%%%%%%%%%%%%%%%%%%%%%%%% UPPER BOUND
\subsection{Upper Bound on the Optimum.}
The upper bound on $F^\ast$ is obtained based on the assumptions from Part (b) of Theorem~\ref{thm:MetaThm}, as follows. 
\begin{restatable}{lemma}{genub}\label{lemma:meta-ub}
If $\prob\Big[\inf_{\bm{\lambda}\in \Delta_M}\Big\|\sum_{i\in M}\lambda_i\bz^i\Big\|_* \leq 4\mu\eps\Big]\leq \gamma$, then 
$$\prob[F^\ast \leq - 2\eps + (\eta - \bar\delta/2)/\mu] \geq 1-\gamma,$$ 
where $F^\ast = \min_{\bx \in \feasset}F(\bx)$ for $F(\bx)$ defined in~\eqref{eqn:rnd_pb}, and ${\cal S}$ is a smoothing map that satisfies the assumptions from Theorem~\ref{thm:MetaThm}. 
\end{restatable}
\begin{proof}
Observe first that:
\begin{align*} 
F^\ast &\leq \frac{1}{\mu}\min_{\bx\in \feasset} {\cal S}\Big(\max\Big\{\frac12\max_{i\in [M]}\big[\innp{ \bz^i,\cdot} - i\bar\delta \big],\|\cdot\|-\frac12(3(1+r)+M\bar\delta)\Big\}\Big) (\bx) \\
&\leq \frac{1}{\mu} \Big(\min_{\bx\in {\cal B}_{\|\cdot\|}} \Big(\max\Big\{\frac12\max_{i\in [M]}\big[\innp{ \bz^i,\bx} - i\bar\delta \big],\|\bx\|-\frac12(3(1+r)+M\bar\delta)\Big\} + \eta\Big) \\
&\leq \frac{1}{2\mu}\Big( \min_{\bx\in {\cal B}_{\|\cdot\|}} \max_{i\in [M]} \innp{ \bz^i, \bx} \Big) +\frac{\eta-\bar\delta/2}{\mu},
\end{align*}
where we have used Property (i) from the definition of local smoothing, and property $(O_1)$ (to assert that the maximum is achieved by the left term).

The rest of the proof is a simple corollary of minimax duality. In particular, as 
$$
\max_{i \in [M]}\innp{\bz^i, \bx} = \max_{\bm{\lambda} \in \Delta_M} \sum_{i\in [M]}\lambda_i \innp{\bz^i, \bx},$$
we have that: $\min_{\bx \in {\cal B}_{\|\cdot\|}}\max_{i \in [M]}\innp{\bz^i, \bx} = \max_{\bm{\lambda}\in \Delta_M} \min_{\bx \in {\cal B}_{\|\cdot\|}} \sum_{i\in [M]}\lambda_i \innp{\bz^i, \bx}$. Finally: 
$$
\min_{\bx \in {\cal B}_{\|\cdot\|}} \sum_{i\in [M]}\lambda_i \innp{\bz^i, \bx} = - \max_{\bx \in {\cal B}_{\|\cdot\|}} \sum_{i\in [M]}\lambda_i \innp{\bz^i, \bx} = - \Big\| \sum_{i=1}^m \lambda_i \bz^i \Big\|_*,
$$
by the (standard) definition of the dual norm $\|\cdot\|_*$. 
Hence:
$$
F^* \leq -\frac{1}{2\mu}\left(\min_{\bm{\lambda}\in\Delta_M}\Big\| \sum_{i=1}^m \lambda_i \bz^i \Big\|_*\right) + \frac{2\eta-\bar\delta}{2\mu},
$$
and it remains to apply the assumption from the statement of the lemma.
\end{proof}
%
%
%%%%%%%%%%%%%%%%%%%%%%%%%%%% LOWER BOUND
\subsection{Lower Bound on the Algorithm's Output.}
Lower bound on the algorithm's output requires more technical work and is based on showing that, at every round $t$, w.h.p., 
the algorithm can only learn $\bz^{1}, \ldots, \bz^t$ and (aside from implicit bounds) has no information about $\bz^{t+1}, \ldots \bz^{M}$. Then, due to Part~(c) of Theorem~\ref{thm:MetaThm}, w.h.p., none of the queried points up to round $M$ can align  well with vector $\bz^{M},$ which will allow us to show that for all the queried points $\bx$ up to round $M$, $F(\bx)$ is $\Omega(\eps)$-far from the optimum $F^*$. 

In the following, we denote the history of the algorithm-oracle interaction until round $t-1$ as $\Pi^{<t}:=(X^s,{\cal O}_F(X^s))_{s< t}$. We also define the following events
$$ 
{\cal E}^t(\bx):=\Big\{\innp{\bz^t, \bx }  > -\frac{\bar\delta}{4}\Big\}\cap 
\Big\{\innp{\bz^i, \bx} < \frac{\bar\delta}{4} \,\,(\forall i>t) \Big\},\,\, \mbox{ and }\,\,
{\bm{\mathcal {E}}}^t := \bigcap_{\bx\in \overline{X}^t} {\cal E}^t(\bx),
$$
where $\bar\delta$ is defined as in Theorem~2. Furthermore, we define the ``good history'' events by:
$$\bm{\mathcal{E}}^{< t} := \bigcap \{\bm{\mathcal{E}}^s: s < t\}.$$
To avoid making vacuous statements, we take ${\bm{\mathcal{E}}}^{<1}$ to be the entire probability space, so that $\prob\big[ \bm{\mathcal{E}}^{<1} \big] = 1.$
We remind the reader that, based on 
Property ($O_2$), when we prove our claim, it suffices to focus on vectors within the ball of radius 4. For this reason, given a batch of queries $X^t=\{\bx_1^t,\ldots,\bx_K^t\}$, we define its {\em relevant queries} as $\overline{X}^t=X^t\cap {\cal B}_{\|\cdot\|}(0,4)$.\\

We first prove that, conditionally on event $\bm{\mathcal{E}}^{<t}$, $X^t$ is a deterministic function of $\{\bz^i\}_{i<t}$.  

\begin{proposition}\label{prop:E-implies-G}
Let $t\in[M-1]$ and suppose event $\bm{\mathcal{E}}^t$ holds. Then, $\forall \bx \in \overline{X}^t+{\cal B}_{\|\cdot\|}(0,r):$
$$
    F(\bx)
    = \frac{1}{\mu}{\cal S}\Big(\max\Big\{\frac12 \max_{i\in[t]}\big[\innp{\bz^i,\cdot} - i\bar\delta\big], \|\cdot\|-\frac12(3(1+r)+M\bar\delta) \Big\} \Big)(\bx).
$$
Moreover, conditionally on
$\bm{\mathcal{E}}^{< t}$,   $X^t$ 
is a deterministic function of $\{\bz^i\}_{i< t}$.
\end{proposition}
\begin{proof}
Let $f(\bx) = \max_{i\in[M]}\big[\innp{\bz^i,\bx} - i\bar\delta\big].$ We will show that for any $\bx_k^t\in \overline{X}^t$ and $\bx$ such that $\|\bx-\bx_k^t\|\leq 2r$, we have 
$f(\bx) = g(\bx),$ where $g(\bx) = \max_{i\in[t]}\big[\innp{\bz^i,\bx} - i\bar\delta\big]$ (notice that $g$ only includes $\bz^i$ for $i \in [t]$).
The  first part of the proposition is then obtained from Part (ii) of Definition~\ref{def:loc_smooth}. %

To prove the claim, notice that since $\|\bz^i\|_{\ast}\leq 1$, we have: 
$$
\innp{\bz^i, \bx} \leq \innp{\bz^i, \bx_k^t} + \|\bx - \bx_k^t\|\cdot\|\bz^i\|_* \leq \innp{\bz^i, \bx_k^t} + 2r.
$$
Similarly, 
$
\innp{\bz^i, \bx_k^t} \leq \innp{\bz^i, \bx} + \|\bx - \bx_k^t\|\cdot\|\bz^i\|_* \leq \innp{\bz^i, \bx} + 2r.
$

Further, by the definition of $\bm{\mathcal{E}}^t$, and since $2r\leq\bar \delta/4$ (by Theorem~2, Assumption (a)), %
\begin{align*}
    \innp{\bz^i, \bx} - i\bar\delta &\,\leq\, \innp{\bz^i, \bx_k^t} + 2r - (t+1)\bar\delta
    \,<\, \frac{\bar \delta}{2} - (t+1)\bar\delta\\
    &\, < \,\innp{\bz^t, \bx_k^t} - 2r - t\delta 
     \,\leq\, \innp{\bz^t, \bx} - t \bar\delta.
\end{align*}

For the second part of the proposition, first observe that $X^t=U^t(\Pi^{<t})$, where $U^t$ is a deterministic function; thus it suffices to prove that, conditionally on $\bm{\mathcal{E}}^{<t}$, $\Pi^{<t}$ is a deterministic function of $\{\bz^i\}_{i<t}$.
We prove the last claim by induction on $t$. For the base case, $\Pi^{<1}$ is empty, thus the property trivially holds. For the inductive step, suppose that conditionally on $\bm{\mathcal{E}}^{<t}$,  $\Pi^{<t}$ is a deterministic function of $\{\bz^i\}_{i<t}$. Now notice that $X^{t}=U^t(\Pi^{<t})$, thus it is a deterministic function of $\{\bz^i\}_{i<t}$. On the other hand, the first part of the proposition guarantees that under $\bm{\mathcal{E}}^t$,
$F|_{\overline{X}^t+{\cal B}_{\|\cdot\|}(0,r)}$ is a deterministic function of $\{\bz^i\}_{i\leq t}$; this proves that $(X^t,{\cal O}_{F}(X^t))$ is a deterministic function of $\{\bz^i\}_{i\leq t}$. Finally, combining this with the induction hypothesis, $\Pi^{<t+1}=(\Pi^{<t},(X^t,{\cal O}_{F}(X^t)))$ is a deterministic function of $\{\bz^{i}\}_{i<t+1}$, proving the inductive step, and thus the result.
\end{proof}

The last result shows that, under $\bm{\mathcal{E}}^{<t}$, $X^t$ is predictable w.r.t.~$\{\bz^i\}_{i<t}$. This means that conditionally on the history and event $\bm{\mathcal{E}}^{<t}$, $X^t$ is fixed. This is key to leverage the randomness of $\{\bz^i\}_{i\geq t}$ for the $t$-th batch of queries. However, there is still a problem: Conditionally on $\bm{\mathcal{E}}^{<t}$, the distribution of $\{\bz^i\}_{i\geq t}$ is different than when there is no conditioning (recall that $\bm{\mathcal{E}}^{<t}$ itself depends on $\{\bz^i\}_{i \geq t}$). In the next lemma we show that, similar as in~\cite{carmon2017lower,nips2018-woodworth}, even after sequential conditioning, the distribution of $\{\bz^i\}_{i\geq t}$ remains sufficiently well-concentrated to carry out the lower bound strategy. 
\begin{restatable}{lemma}{events}\label{lemma:gen-lb-event}
Under Assumptions (a) and (c) from Theorem~2, we have:
$$
\prob\Big[\bigcap_{t\in[M]}\bm{\mathcal{E}}^t \Big] \geq 1-\gamma.
$$
\end{restatable}
\begin{proof}%
First observe that, for any $1\leq t \leq M$, by the law of total probability: 
\begin{eqnarray}
\prob[(\bm{\mathcal{E}}^t)^c|\, \bm{\mathcal{E}}^{<t}]
&=& \int_{\bxi} \prob[(\bm{\mathcal E}^t)^c|\,\bm{\mathcal E}^{<t},\,\{\bz^i\}_{i<t}=\bxi]
\,\dd\prob\big[\{\bz^i\}_{i<t}=\bxi|\bm{\mathcal E}^{<t}\big] \notag
\end{eqnarray}

On the other hand, by the previous proposition, $X^{t}$ is a deterministic function of $\{\bz^i\}_{i<t}$, conditionally on $\bm{\mathcal E}^{<t}$. Recall that:
$$
(\bm{\mathcal{E}}^t)^c = \Big\{\exists \bx\in \overline{X}^t:\,\innp{ \bz^t, \bx }  < -{\bar\delta}/{4} \mbox{ or } (\exists i>t) \innp{ \bz^i, \bx }  > {\bar\delta}/{4}\Big\}.
$$
To simplify the notation, denote:
$$
(\bm{\mathcal{E}}^t)^c_{\{\overline{X}^t \rightarrow {X}\}} = \Big\{\exists \bx\in {X}:\,\innp{ \bz^t, \bx }  < -{\bar\delta}/{4} \mbox{ or } (\exists i>t) \innp{ \bz^i, \bx }  > {\bar\delta}/{4}\Big\}.
$$
Therefore, we further have: 
\begin{align*}
    \prob[(\bm{\mathcal{E}}^t)^c|\, \bm{\mathcal{E}}^{<t}] \leq &
    {\int_{\bxi} %
    \sup_{\substack{X\subseteq {\cal B}_{\|\cdot\|}(0,4),\\ |X|\leq K}}\prob\Big[(\bm{\mathcal{E}}^t)^c_{\{\overline{X}^t \rightarrow {X}\}}\Big| \bm{\mathcal{E}}^{<t}, \{\bz^i\}_{i<t}=\bxi \Big]
    \dd\prob\big[\{\bz^i\}_{i<t}=\bxi|\bm{\mathcal E}^{<t}\big]},
\end{align*}
where we have used that $\overline{X}^t$ is conditionally deterministic. 
Now that $X$ is fixed, we can use the union bound as follows:
\begin{align}
    \prob[&(\bm{\mathcal{E}}^t)^c|\, \bm{\mathcal{E}}^{<t}]\notag\\
    \leq &{ K} \int_{\bxi} \sup_{\bx \in {\cal B}_{\|\cdot\|}(0, 4)}\prob\Big[\innp{\bz^t, \bx} < -\bar{\delta}/4\Big| \bm{\mathcal{E}}^{<t}, \{\bz^i\}_{i<t}=\bxi \Big] \dd\prob[\{\bz^i\}_{i<t}=\bxi|\bm{\mathcal E}^{<t}]\notag\\
    &+ {(M-1)K} \int_{\bxi} \sup_{\bx \in {\cal B}_{\|\cdot\|}(0, 4)}\max_{j > t}\prob\Big[\innp{\bz^j, \bx} > \bar{\delta}/4\Big| \bm{\mathcal{E}}^{<t}, \{\bz^i\}_{i<t}=\bxi \Big] \dd\prob[\{\bz^i\}_{i<t}=\bxi|\bm{\mathcal E}^{<t}].\notag
\end{align}
Observe for the first integral in the last expression that we can write:
\begin{align*}
    \int_{\bxi} \sup_{\bx \in {\cal B}_{\|\cdot\|}(0, 4)} & \prob\Big[\innp{\bz^t, \bx} < -\bar{\delta}/4\Big| \bm{\mathcal{E}}^{<t}, \{\bz^i\}_{i<t}=\bxi \Big] \dd\prob[\{\bz^i\}_{i<t}=\bxi|\bm{\mathcal E}^{<t}]\\
    & = \int_{\bxi} \sup_{\bx \in {\cal B}_{\|\cdot\|}(0, 4)} 
    \frac{\prob\big[\innp{\bz^t, \bx} < -\bar{\delta}/4, \bm{\mathcal{E}}^{<t}| \{\bz^i\}_{i<t}=\bxi \big]}{\prob\big[\bm{\mathcal{E}}^{<t} | \{\bz^i\}_{i<t}=\bxi\big]} \dd\prob\big[\{\bz^i\}_{i<t}=\bxi|\bm{\mathcal E}^{<t}\big]\\
    &= \int_{\bxi} \sup_{\bx \in {\cal B}_{\|\cdot\|}(0, 4)} 
    \frac{\prob\big[\innp{\bz^t, \bx} < -\bar{\delta}/4, \bm{\mathcal{E}}^{<t}| \{\bz^i\}_{i<t}=\bxi \big]}{\prob\big[\bm{\mathcal{E}}^{<t}\big]} \dd\prob\big[\{\bz^i\}_{i<t}=\bxi\big]\\
    &\leq \int_{\bxi} \sup_{\bx \in {\cal B}_{\|\cdot\|}(0, 4)} 
    \frac{\prob\big[\innp{\bz^t, \bx} < -\bar{\delta}/4| \{\bz^i\}_{i<t}=\bxi \big]}{\prob\big[\bm{\mathcal{E}}^{<t}\big]} \dd\prob\big[\{\bz^i\}_{i<t}=\bxi\big]\\
    & = \frac{\sup_{\bx \in {\cal B}_{\|\cdot\|}(0, 4)} 
    \prob\big[\innp{\bz^t, \bx} < -\bar{\delta}/4\big]}{\prob\big[\bm{\mathcal{E}}^{<t}\big]},
\end{align*}
where we have used the Bayes rule in the second equality. Applying the same arguments to the second integral, we finally have:
\begin{eqnarray*}
\prob\big[(\bm{\mathcal{E}}^t)^c|\, \bm{\mathcal{E}}^{<t}\big] &\leq& \frac{MK\sup_{\bx \in {\cal B}_{\|\cdot\|}(0, 4)} \max\big\{
    \prob\big[\innp{\bz^t, \bx} < -\bar{\delta}/4\big], \, \max_{j > t}\prob\big[\innp{\bz^j, \bx} > \bar{\delta}/4\big]\big\}}{\prob[\bm{\mathcal{E}}^{<t}]}\\
&\leq& \frac{MK\exp\{-\alpha\bar\delta^2/256\}}{\prob[\bm{\mathcal{E}}^{<t}]}
\leq \frac{\gamma}{\prob[\bm{\mathcal{E}}^{<t}]}.
\end{eqnarray*}
Inductively, each  $\bm{\mathcal{E}}^{<t}$ happens with non-zero probability, as $\prob\big[\bm{\mathcal{E}}^{<1}\big] = 1$ and $\gamma < 1.$ 

We conclude the proof by conditioning:
$$ 
\prob\Big[ \bigcap_{t\in[M]} \bm{\mathcal{E}}^t\Big]
=\frac{\prob\Big[ \bigcap_{t\in[M]} \bm{\mathcal{E}}^t\Big]}{\prob\Big[ \bigcap_{t<M} \bm{\mathcal{E}}^t\Big]}
\prob\Big[ \bigcap_{t<M} \bm{\mathcal{E}}^t\Big]
= \prob\Big[\bm{\mathcal{E}}^M \Big|\, \bm{\mathcal{E}}^{<M} \Big]
\, \prob\Big[ \bm{\mathcal{E}}^{<M} \Big]
\geq 1-\gamma.
$$
\end{proof} 

Finally, Lemma~\ref{lemma:gen-lb-event} and Proposition~\ref{prop:E-implies-G} imply the following lower bound on the algorithm's output:
\begin{equation}\label{eq:gen-F-lb}
    \prob\Big[\min_{t \in [M],\, k \in [K]}F(\bx^t_k) \geq -\frac{\bar\delta}{2\mu}\Big(\frac{1}{4} + M + \frac{2\eta}{\bar\delta}\Big)\Big] \geq 1-\gamma,
\end{equation}
as, when all events $\{\bm{\mathcal{E}}^t:\, t\in[M]\}$ hold simultaneously (and, in particular, when event $\bm{\mathcal{E}}^M$ holds), we have, by the definitions of these events and the random problem instance~\eqref{eqn:rnd_pb}, that:
$$
\min_{t \in [M],\, k \in [K]} F(\bx_k^t)\geq\frac{1}{2\mu} \min\Big\{ \innp{ \bz^M, \bx}-M\bar\delta:\,
\bx\in \cup_{t\in[M]}\overline{X}^t \Big\} -\frac{\eta}{\mu}
\geq -\frac{\bar\delta}{8\mu}-M\frac{\bar\delta}{2\mu}-\frac{\eta}{\mu}.
$$
%
%
%%%%%%%%%%%%%%%%%%%%%%%%%%%% OPTIMALITY GAP
\subsection{Bounding the Optimality Gap.}
To complete the proof of Theorem~\ref{thm:MetaThm}, it remains to combine the results from the previous two subsections and argue that, w.p.~$1-\gamma,$ the optimality gap of any solution output by the algorithm is higher than $\eps.$% as follows.

\noindent
\textbf{Remaining Proof of Theorem~\ref{thm:MetaThm}}
Applying Lemma~\ref{lemma:meta-ub}, with probability $1-\gamma,$ $F^* \leq -2\eps + (\eta-\bar\delta/2)/\mu.$ From Eq.~\eqref{eq:gen-F-lb}, w.p. $1-\gamma$, $\min_{t \in [M],\, k \in [K]}F(\bx^t_k) \geq -\frac{\bar\delta}{2\mu}\Big(\frac{1}{4} + M + \frac{2\eta}{\bar\delta}\Big)$. Hence, with probability $1-2\gamma,$
\begin{align*}
    \min_{t \in [M],\, k \in [K]}F(\bx^t_k) - F^* \geq 2\eps - \frac{\bar\delta}{2\mu}\Big(M-\frac{3}{4}\Big) - \frac{2\eta}{\mu} > \eps,
\end{align*}
as, by the theorem assumptions, $\bar\delta \leq \eps\mu/M$ and $\eta \leq \eps\mu/4.$
\hfill $\square$ 
%
%
%
%%%%%%%%%%%%%%%%%%%%%%%%%%% OMITTED PROOFS FROM APPS
\section{Omitted Proofs from Section~\ref{sec:apps}}\label{app:proofs-apps}
\subsection{Nonsmooth Optimization for $1\leq p \leq 2$}
\begingroup
\def\thetheorem{\ref{lemma:1<p<2-ub}}
\lemoneptwoub*
\addtocounter{theorem}{-17}
\endgroup

\begin{proof}
By the choice of $\bz^i$'s, $\|\sum_{i\in [M]}\lambda_i \bz^i\|_{p^*}^{p^*} = \frac{1}{d}\sum_{j \in [d]}\left|\sum_{i \in [M]}\lambda_i r_j^i\right|^{p^*}.$ Hence, using  Lemma~\ref{lemma:eps-net-for-ub}, it suffices to show that:
$$
\prob \Big[\frac{1}{d}\sum_{j \in [d]}\Big|\sum_{i \in [M]}\lambda_i r_j^i\Big|^{p^*} \leq (\eps')^{p^*} \Big] \leq \gamma',
$$
for $\eps' = 5\eps$ and sufficiently small $\gamma'$ (namely, for $\gamma' \leq (\frac{\eps}{3})^M \gamma$).

Let $Y_j:=\Big|\sum_{i\in[M]} \lambda_i r_j^i\Big|^{p^*}$, for $j\in[d]$, and notice that $Y_j$'s are nonnegative and i.i.d. 
Moreover, by Khintchine's Inequality, there exist constants $c,c^{\prime}$ such that:
\begin{eqnarray*}
\EE[Y_1] &=&\EE\Big|\sum_{i\in[M]} \lambda_i r_j^i\Big|^{p^*}\geq c \Big(\sum_{i\in[M]} \lambda_i^2\Big)^{p^*/2} = c \|\bm{\lambda}\|_2^{p^*} \\
\EE[Y_1^2] &=&\EE\Big|\sum_{i\in[M]} \lambda_i r_j^i\Big|^{2p^*}\leq c^{\prime} \Big(\sum_{i\in[M]} \lambda_i^2\Big)^{2p^*/2} = c'\|\bm{\lambda}\|_2^{2p^*}.
\end{eqnarray*}
In particular, when $1\leq p\leq 2,$ $c' = 1$ and $c \geq  1/\sqrt{2}$~\cite{haagerup1981best}. 
Therefore, by the left-sided Bernstein's Inequality (Theorem~\ref{thm:LeftBernstein}) 
for any $0 < \eta < c:$
\begin{eqnarray*}
\prob\Big[ \frac{1}{d}\sum_{j\in [d]} Y_j \leq (c-\eta) \|\bm{\lambda}\|_2^{p^*}\Big]
&\leq& \exp\left( -\dfrac{d\eta^2\|\bm{\lambda}\|_2^{2p^*}}{2c^{\prime}\|\bm{\lambda}\|_2^{2p^*}} \right)
=\exp\left( -\frac{d\eta^2}{2c^{\prime}} \right).
\end{eqnarray*}
As $\bm{\lambda} \in \Delta_M,$ it must be $\|\bm{\lambda}\|_2 \geq 1/\sqrt{M}$. Choosing $\eta = c/2,$ we have $(c-\eta)\|\bm{\lambda}\|_2\geq \frac{1}{2\sqrt{2M}}\geq 5\eps = \eps',$ and it follows that:
$$ 
\prob\Big[ \frac{1}{d}\sum_{j\in [d]} \Big|\sum_{i\in[M]} \lambda_i r_j^i\Big|^{p^*} \leq (\varepsilon')^{p^*}\Big]
\leq \exp\Big( -\frac{d c^2}{8c^{\prime}} \Big)\leq \exp\Big( -\frac{d}{8\sqrt{2}}\Big).
$$
To complete the proof, it suffices to have $d \geq 8\sqrt{2}\left(\ln(1/\gamma) + M \ln(3/\eps)\right).$ This is clearly satisfied for $M \leq \frac{d/12 - \ln(1/\gamma)}{\ln(3/\eps)}$ from the lemma's assumptions.
\end{proof}

\nonstanoneptwo*
\begin{proof} 
As before, we will prove this result as an application of Theorem~\ref{thm:MetaThm}. Let $\br^i$ be an independent standard Rademacher sequence in $\RR^d$ and $\bz^i= \frac{1}{d^{1/p^*}}\br^i$. Assumption (a) is automatically satisfied since $\kappa=0$. On the other hand, and since we are working on a feasible set $\feasset\neq {\cal B}_p^d$, we need to adapt the upper bound on the optimum provided in Assumption (b). 
By standard duality arguments:
\begin{align*}
    \min_{\bx \in {\cal B}_2(1/d^{1/p-1/2})}\max_{i \in [M]} \{\innp{\bz^{i}, \bx} - i \delta \} &\leq \frac{1}{d^{1/p-1/2}}\min_{\bx \in {\cal B}_2(1)}\max_{\blambda \in \Delta_M} \Big\langle\sum_{i\in[M]}\lambda_i\bz^{i}, \bx\Big\rangle -  \delta\\
    &= -\frac{1}{d^{1/p-1/2}} \min_{\blambda \in \Delta_M} \Big\| \sum_{i\in [M]} \lambda_i \bz^i \Big\|_2 - \delta.
\end{align*}
Therefore, we replace Property (b) by the following probabilistic guarantee
$$\prob\Big[ \min_{\blambda \in \Delta_M} \Big\| \sum_{i\in [M]} \lambda_i \frac{\bz^i}{d^{1/p-1/2}} \Big\|_2 \leq 4\varepsilon \Big] = \prob\Big[ \min_{\blambda \in \Delta_M} \Big\| \sum_{i\in [M]} \lambda_i \frac{\br^i}{\sqrt d} \Big\|_2 \leq 4\varepsilon \Big] \leq \gamma$$
for any $M\leq \min\Big\{\frac{1}{200\varepsilon^2},\frac{d/12-\ln(1/\gamma)}{\ln(3/\varepsilon)}\Big\}$, which holds by Lemma~\ref{lemma:1<p<2-ub}. 
On the other hand, the concentration required in Assumption (c) is satisfied for any $\bx\in{\cal B}_2^d(1/{d}^{1/p-1/2})$ by Hoeffding
$$ \prob[\langle \bz^i,\bx\rangle >\delta] \leq \exp\{-\delta^2 d^{2/p^*}/\|\bx\|_2\} \leq \exp\{-d\delta^2\}.$$
In particular, we can choose $\alpha=d$. Finally, Assumption (d) is satisfied for $M\leq \frac{\varepsilon}{\delta}=\frac{\varepsilon}{16}\sqrt{\frac{d}{\ln(M K/\gamma)}}$, completing the proof.
\end{proof}
\subsection{Smooth, Weakly Smooth, and Nonsmooth Optimization for $p \geq 2$}
We start by showing that, under suitable constraints on $M$ and $L,$  Assumptions (b) and (c) from Theorem~\ref{thm:MetaThm} are satisfied. This will suffice to apply Theorem~\ref{thm:MetaThm} in the case of nonsmooth optimization (i.e., for ${\cal S}$ being a $(0, 0, 0, 1)$-local smoothing). To obtain results in the smooth and weakly smooth settings, we will then show how to satisfy the remaining assumptions for a suitable local smoothing.

In terms of Assumption (b), we can in fact obtain a much stronger result than needed in Theorem~\ref{thm:MetaThm}: 
\addtocounter{theorem}{16}
\begin{lemma}\label{lemma:p>2-ub} 
Let $p \geq 2$, $\varepsilon\in(0,1),$ $\mu > 0,$ $\bz^i$'s chosen as described in Section~\ref{sec:p>2} and:
$$
M \leq \Big(\frac{1}{4\mu\eps}\Big)^p
$$
then: 
$$
\prob\Big[ \min_{\bm{\lambda}\in \Delta_M}\Big\|\sum_{i \in [M]}\lambda_i\bz^i\Big\|_{p^*} \leq 4\mu\eps \Big] = 0.
$$
\end{lemma}
\begin{proof}
Let $\bm{\lambda} \in \Delta_M$ be fixed. Observe that, since $\bz^i$'s have disjoint support (each $\bz^i$ is supported on $J_i$ such that $|J_i| = L$ and $J_i \cap J_{i'} = \emptyset$ for all $i \neq i'$), vector $\sum_{i\in[M]}\lambda_i \bz^i$ is such that its coordinates indexed by $j \in J_i$ ($L$ of them) are equal to $\lambda_i z^i_j$, $\forall i \in [M]$. Therefore, using the definition of $\bz^i$ (Equation~\ref{eq:zi-def}):
\begin{equation*}
    \Big\|\sum_{i\in [M]} \lambda_i \bz^i \Big\|_{p^*}^{p^*} = \sum_{i\in [M]} \Big(L \cdot \Big(\lambda_i \frac{1}{L^{1/{p^*}}}\Big)^{p^*}\Big) = \|\bm{\lambda}\|_{p^*}^{p^*}.
\end{equation*}
By the relationship between $\ell_p$ norms and the definition of $\bm{\lambda},$ we have that $1 = \|\bm{\lambda}\|_1 \leq M^{1/p}\|\bm{\lambda}\|_{p^*}$. Hence:
$$
\Big\|\sum_{i\in [M]} \lambda_i \bz^i \Big\|_{p^*} = \|\bm{\lambda}\|_{p^*} \geq M^{-1/p} \geq 4\mu\eps. 
$$
Since this holds for all $\bm{\lambda}\in\Delta_M$ surely, the proof is complete.
\end{proof}

For Assumption (c), we have the following (simple) lemma:
\begin{lemma}\label{lemma:p>2-lb} 
Let $p \geq 2$ and $\bz^i$'s chosen as described in Section~\ref{sec:p>2}, then:
$$
\prob[\innp{\bz^i, \bx} \geq \delta] = \prob[\innp{\bz^i, \bx} \leq - \delta] \leq \exp\Big(-\frac{L\delta^2}{2}\Big) \qquad (\forall \bx \in {\cal B}_p^d).
$$
\end{lemma}
\begin{proof}
By the definition of $\bz^i$ and Hoeffding's Inequality, $\forall \bx \in \feasset:$
\begin{align*}
    \prob[\innp{\bz^i, \bx} > \delta] = \prob[\innp{\bz^i, \bx} < -\delta] &= \prob\Big[\sum_{j \in J_i}r_j^i x_j > \delta L^{1/p^*} \Big]\\
    & \leq \exp \Big( - \frac{L^{2/p^*} \delta^2}{2 \sum_{j \in J_i}{x_j}^2} \Big).
\end{align*} 
As $|J_i| = L,$ by the relations between $\ell_p$ norms, $(\sum_{j \in J_i}{x_j}^2)^{1/2} \leq L^{1/2 - 1/p} (\sum_{j \in J_i}{x_j}^p)^{1/p} \leq L^{1/2 - 1/p}$. Thus, it follows that:
\begin{align*}
    \prob[\innp{\bz^i, \bx} > \delta] = \prob[\innp{\bz^i, \bx} < -\delta] 
    & \leq \Big( - \frac{L^{2/p^*} \delta^2}{2 L^{1-2/p}} \Big) = \exp\Big(- \frac{L\delta^2}{2}\Big),
\end{align*}
as claimed.
\end{proof}
To obtain the result for the nonsmooth case, we can take $\mu = 1$ and apply Theorem~\ref{thm:MetaThm}, as follows.
\ptwothmnsmooth*
\begin{proof}
For Lemma~\ref{lemma:p>2-ub} to apply, it suffices to have $M \leq \frac{1}{(4\eps)^p},$ as in the nonsmooth case $\mu = 1.$ Lemma~\ref{lemma:p>2-lb} implies that it suffices to set $\alpha = L/2 = d/(2M).$ As ${\bar \delta} = 16 \sqrt{\frac{\ln(MK/\gamma)}{\alpha}}$, to satisfy Assumption (d) from Theorem~\ref{thm:MetaThm} (which requires ${\bar\delta} \leq \eps/M$), it suffices to have:
$$
M \leq \frac{\eps}{16}\sqrt{\frac{d}{2 M\ln(MK/\gamma)}},
$$
or, equivalently: $M \leq \frac{\eps^{2/3}}{8}\left(\frac{d}{\ln(MK/\gamma)}\right)^{1/3},$ as claimed.
\end{proof}

To obtain lower bounds for the $\kappa$-weakly smooth case (where $\kappa \in [0, 1];$ $\kappa = 0$ is the nonsmooth case from the above and $\kappa = 1$ is the standard notion of smoothness), we need to, in addition to using Lemmas~\ref{lemma:p>2-ub} and~\ref{lemma:p>2-lb}, choose an appropriate local smoothing that satisfies the remaining conditions from Theorem~\ref{thm:MetaThm}. By doing so, we can obtain the following result.
\ptwothmsmooth*
\begin{proof} 
From Remark~\ref{rem:smoothings} we have that
$\ell_p^d$ is $(\kappa, \eta, \eta, \mu)$-locally smoothable (observe here that $r =\eta$) for any $0\leq \kappa\leq 1$, as long as $\mu = 2^{1-\kappa}(\min\{p ,\ln(d)\}/\eta)^{\kappa}$. 

Let $\eta = {\bar\delta}/8.$ Denote ${\bar\mu} = 2^{1-\kappa}(8\min\{p ,\ln(d)\})^{\kappa},$ so that $\mu = \frac{{\bar\mu}}{{\bar\delta}^\kappa}.$ To satisfy Assumptions (a) and (d), we need to have ${\bar\delta} \leq \min\{2\eps\mu, \eps\mu/M\},$ and it suffices to enforce $M \leq \frac{\mu\eps}{{\bar\delta}} = \frac{\eps {\bar\mu}}{{\bar\delta}^{1+\kappa}}$. To satisfy Assumption (c), 
by Lemma~\ref{lemma:p>2-lb} we can choose $\alpha = \frac{L}{2}$, which leads to the following bound on $M$:
\begin{equation}\label{eq:p>2-bnd-on-M-1}
M \leq \frac{\eps{\bar\mu}}{2^{4(1+\kappa)}}\Big(\frac{L}{2\ln(MK/\gamma)}\Big)^{\frac{1+\kappa}{2}}.
\end{equation} 
To satisfy the remaining assumption from Theorem~\ref{thm:MetaThm} (Assumption (b), using Lemma~\ref{lemma:p>2-ub}), we need to impose the following constraint on $M$:
\begin{equation}\label{eq:p>2-bnd-on-M-2}
M \leq \left(\frac{1}{4\mu\eps}\right)^{p}=\left(\frac{{\bar{\delta}}^\kappa}{4\eps {\bar\mu}}\right)^p =
\Big(\frac{2^{2(2\kappa-1)}}{\varepsilon{\bar\mu}}\Big)^p
\Big(\frac{L}{2\ln(MK/\gamma)}\Big)^{-\frac{p\kappa}{2}}
\end{equation}
The right-hand sides of the inequalities in Equations~\eqref{eq:p>2-bnd-on-M-1} and~\eqref{eq:p>2-bnd-on-M-2} are equal when 
$$
L = 2^9\ln(MK/\gamma)\cdot \left(\frac{4}{(4\bar\mu\varepsilon)^{p+1}}\right)^{2/[1+\kappa(p+1)]}
$$
and, thus, we make this choice for $L.$ As $d \geq ML,$ we also need to satisfy $M \leq d/L,$ finally leading to the claimed bound:
\begin{equation*}
    M \leq \min\left\{\bigg(\dfrac{1}{4^{1+\kappa}\bar\mu\varepsilon}\bigg)^{\frac{p}{1+\kappa(1+p)}},\; \frac{d}{2^9\ln(MK/\gamma)}\left(4^{\frac{p}{1+p}}{\bar\mu}\eps\right)^{\frac{2(1+p)}{1+\kappa(1+p)}}\right\}
\end{equation*}
The rest of the proof follows by plugging ${\bar\mu} = 2^{1+2\kappa}(\min\{p, \ln(d)\})^\kappa$ in the last equation. % 
\end{proof}
\subsection{Smooth and Weakly Smooth Optimization for $1\leq p < 2$}
\oneptwothmsmooth*
\begin{proof} \textbf{Sketch}
By Dvoretzky's Theorem (see Appendix~\ref{app:background}), there exists a universal constant $\nu>0$ 
such that for any $T\leq \nu d$ there exists a subspace $F\subseteq \RR^d$ of dimension $T$, and  
a centered ellipsoid ${\cal E}\subseteq F$, such that 
\begin{equation} \label{eqn:Dvoretzky} \frac12 {\cal E} \subseteq F\cap {\cal B}_p^d \subseteq {\cal E}.\end{equation}
By an application of the Hahn-Banach theorem, we can certify that there exist vectors 
$\bg^1,\ldots,\bg^T\in {\cal B}_{p^*}^d$, such that
${\cal E}=\{ \bx\in F:\,\, \sum_{i\in [T]} \innp{ \bg^i,\bx}^2\leq 1\}.$

Consider now linear mapping $G:(\RR^d,\|\cdot\|_p)\mapsto (\RR^T,\|\cdot\|_{\infty})$ 
such that $G\bx:=(\innp{ \bg^1,\bx},\ldots,\innp{ \bg^T,\bx})$, and notice that by the  previous paragraph the operator norm of $G$ is upper bounded by 1. We observe that:
\begin{itemize}
\itemsep0em 
\item For any $f\in {\cal F}_{\ell_{\infty}^T}^{\kappa}(\mu)$, function $\tilde f:=f\circ G$ belongs to 
${\cal F}_{\ell_p^d}^{\kappa}(\mu)$. In other words, the whole function class ${\cal F}_{\ell_{\infty}^T}^{\kappa}(\mu)$
can be obtained from ${\cal F}_{\ell_p^d}^{\kappa}(\mu)$ through the linear embedding $G$.
\item We claim that any local oracle for the class $\{ \tilde f:\,\, f\in {\cal F}_{\ell_{\infty}^T}^{\kappa}(\mu)\}$
can be obtained from a local oracle for the class  ${\cal F}_{\ell_p^d}^{\kappa}(\mu)$
(for a proof of this claim, see \cite[Appendix C]{Guzman:2015}).
\item From \eqref{eqn:Dvoretzky}, the set ${\cal Y}=G{\cal B}_p^d$ is such that
$ \frac{1}{2\sqrt{T}} {\cal B}_{\infty}^T\subseteq\frac12 {\cal B}_2^T \subseteq {\cal Y} \subseteq {\cal B}_2^T.$
\end{itemize}
From these facts, we can conclude that the oracle complexity over $\feasset$ with function class ${\cal F}_{\ell_p^d}^{\kappa}(1)$ is at least the one obtained in the embedded space ${\cal Y}$ with the respective embedded function class ${\cal F}_{\ell_{\infty}^T}^{\kappa}(1)$, thus
\begin{align*}
\mbox{Compl}_{\mbox{\footnotesize HP}}^{\gamma}({\cal F}_{\ell_{p}^d}^{\kappa}(1),\feasset,K,\varepsilon)
&\geq \mbox{Compl}_{\mbox{\footnotesize HP}}^{\gamma}({\cal F}_{\ell_{\infty}^T}^{\kappa}(1),{\cal Y},K,\varepsilon)\\
&\geq \, \mbox{Compl}_{\mbox{\footnotesize HP}}^{\gamma}({\cal F}_{\ell_{\infty}^T}^{\kappa}(1),{\cal B}_{\infty}^T(0,1/[2\sqrt T]),K,\varepsilon)
\end{align*}
Denote $\eps' = 2\eps\sqrt{T}$. By Theorem~\ref{thm:p>2-smooth} applied to $p=
\infty$, together with Remark~\ref{rem:rescaling_LB}, we get that it is sufficient to require, as long as $T \leq \nu d,$ that:
\begin{align*}
    M = \min\bigg\{& \frac{1}{\ln(T)}\Big(\frac{1}{2^{3+4\kappa}\eps'}\Big)^{1/\kappa},\; \frac{T\ln^2(T)}{2^9\ln(\nu d K/\gamma)}\big(2^{3+2\kappa}\eps'\big)^{2/\kappa}  \bigg\}\\
    = \min\bigg\{& \frac{1}{\ln(T)}\Big(\frac{1}{2^{4(1+\kappa)}\eps\sqrt{T}}\Big)^{1/\kappa}, \;  \frac{T\ln^2(T)}{2^9\ln(\nu d K/\gamma)}\big(2^{2(2+\kappa)}\eps\sqrt{T}\big)^{2/\kappa}\bigg\}.
\end{align*}
In the last expression, the left term in the minimum is lower whenever:
$$
T \ln^2 T \geq (2^9\ln(d K/\gamma))^{\frac{2\kappa}{3+2\kappa}} \Big(\frac{1}{2}\Big)^\frac{8(2+3\kappa)}{3+2\kappa} \Big(\frac{1}{\eps}\Big)^{\frac{6}{3+2\kappa}},
$$
and it suffices to choose:
$$
T = \left\lceil 2(\ln(\nu d K/\gamma))^{\frac{2\kappa}{3+2\kappa}}\Big(\frac{1}{\eps}\Big)^{\frac{6}{3+2\kappa}}\right\rceil.
$$
Under this choice, as long as $d \geq T/\nu$, the oracle complexity is lower bounded by:
$$
M = \frac{c_\kappa}{\ln(1/\eps) + \kappa \ln\ln(dK/\gamma)}\Big(\frac{1}{\eps}\Big)^{\frac{2}{3+2\kappa}},
$$
where $c_\kappa$ is an absolute constant that only depends on $\kappa,$ as claimed.
\end{proof}
\end{document}